\documentclass{amsart}
\usepackage{geometry}
\usepackage[english]{babel}
\usepackage[latin1]{inputenc}
\usepackage[sc]{mathpazo}
\usepackage[colorlinks]{hyperref}
\usepackage{pgfplots}
\pgfplotsset{compat=1.10}
\usepgfplotslibrary{fillbetween}
\usetikzlibrary{patterns}

\usepackage{todonotes}

\usepackage{amssymb,amsmath,amsthm,amsfonts, mathrsfs, mathtools}
\usepackage{graphicx}
\usepackage{enumitem}

\def\R {\mathbb{R}}
\def\N {\mathbb{N}}
\def\C {\mathcal{C}}
\def\eps{\varepsilon}
\def\utss{u.t.s.s.}
\def\aa{A}

\renewcommand{\div}{\mathrm{div}}

\DeclareMathOperator{\tr}{tr}

\newtheorem{proposition}{Proposition}[section]
\newtheorem{theorem}[proposition]{Theorem}
\newtheorem*{theorem*}{Theorem}

\newtheorem{lemma}[proposition]{Lemma}

\theoremstyle{definition}

\newtheorem{remark}[proposition]{Remark}
\numberwithin{equation}{section}

\title[Predator-prey (II): uniform estimates]{Predator-prey models with competition, Part II: \\ uniform regularity estimates}
\author{Henri Berestycki}
\email{hb@ehess.fr}
\address{\'{E}cole des Hautes \'{E}tudes en Sciences Sociales, PSL Research University Paris, Centre d'analyse et de math\'{e}matique sociales (CAMS), CNRS, 54 bouvelard Raspail, 75006, Paris}

\author{Alessandro Zilio}
\email{azilio@math.univ-paris-diderot.fr}
\address{Universit\'{e} Paris Diderot - Paris 7, Laboratoire J.-L.\ Lions (CNRS UMR 7598), Paris, France, 8 place Aur\'elie Nemours, 75205, Paris CEDEX 13}


\begin{document}

\begin{abstract}
We study a system of elliptic equations with strong competition and an arbitrary large number of components. The system is related to a model of predators and prey, with a single and where several predators compete with each other. In this paper we derive regularity estimates of the solutions that are independent of the number of components (i.e., groups of predators) and the strength of competition between the components. 
\end{abstract}

\maketitle

\section{Introduction}\label{sec unif bounds}

\nocite{BZ_ecology}


We study the regularity of positive solutions $\mathbf{v} = (w_1, \dots, w_N, u)$ of the system
\begin{equation}\label{eqn model k}
	\begin{cases}
		- d_i \Delta w_i  = \left(- \omega_i + k_i u - \beta \sum_{j \neq i} a_{ij} w_j\right) w_i &\text{in $\Omega$}\\
		- D \Delta u = \left(\lambda - \mu u - \sum_{i=1}^N k_i w_i \right)u &\text{in $\Omega$}\\
		\partial_\nu w_i = \partial_\nu u = 0 &\text{on $\partial \Omega$}.
	\end{cases}
\end{equation}
We are chiefly interested in estimates that are independent both on the competition term $\beta$ and on the number of densities $N$. 
For this reason, we will work under the following uniform assumption. We assume that there exists $\delta \in (0,1)$ (fixed throughout the paper) such that 
\begin{equation}\label{unifass}
\begin{gathered}
	\delta \leq \lambda, \, \mu, \, d_i, \,  \omega_i, \,  k_i, \,   a_{ij} \leq \frac{1}{\delta}\\
	\lambda k_i - \mu \omega_i > \delta
\end{gathered}
\end{equation}
for any choice of the parameters in \eqref{eqn model k}. The assumption $\lambda k_i - \mu \omega_i > \delta$, although not necessary, is justified by the fact that if there exists $i \in \N$ such that $\lambda k_i - \mu \omega_i \leq 0$, it can be shown (\cite[Lemma 2.1]{BerestyckiZilio_PI}) that the component $w_i$ is necessarily equal to $0$. Some of the inequalities we are about to prove can be derived more easily in the case when $N$ is fixed. Here we derive these inequalities in the more general case of an arbitrary number of densities. This feature renders the derivation of the estimates considerably more delicate. This aspect has not been considered hitherto for this type of systems. 

The main results of this paper are contained in the following statements. We begin by studying the regularity of the solutions, uniformly in $\beta$ and $N$.

\begin{theorem}\label{prp asymptotic k}
Let $\Omega \subset \R^n$ be a smooth domain. Let $\beta \geq 0$ and $N \in \N$. We consider any set of positive parameters $D$, $d_i$, $\omega_i$, $k_i$, $a_{ij} = a_{ji}$ for $1 \leq i,j \leq k$ that satisfy the uniform assumption \eqref{unifass}. We consider a non negative solution $\mathbf{v} = (w_1, \dots, w_N, u) = (\mathbf{w},u)\in C^{2,\alpha}(\Omega)$ of system \eqref{eqn model k}. Then  all components of $\mathbf{v}$ are uniformly bounded in $L^\infty(\Omega)$ with respect to $\beta > 0$ and $N \in \N$, and there exists a constant $C = C(\delta, \Omega) >0$ (in particular $C$ is independent of $\beta$ and $N$) such that
\[
  0 \leq u \leq \frac{\lambda}{\mu}, \qquad \text{and} \qquad 0 \leq \sum_{i=1}^N w_i \leq C.
\]
Moreover, for any $\alpha \in (0,1)$, there exists $C_\alpha = C(\alpha, \delta, \Omega)$ (independent of $\beta$ and $N$) such that
\[
	\| u \|_{\C^{2,\alpha}(\overline \Omega)} \leq C_\alpha
\]
and
\[
  \max_{i \in \{1, \dots, N\}} \| w_i \|_{C^{0,\alpha}(\overline{\Omega})} + \max_{ \delta \leq \pi_i \leq 1/	\delta} \left\| \sum_{1=1}^N \pi_i w_i \right\|_{C^{0,\alpha}(\overline{\Omega})} \leq C_\alpha \left\| \sum_{1=1}^N w_i \right\|_{L^{\infty}(\overline{\Omega})}.
\]
for any positive reals $\pi_1, \dots, \pi_N$ with $\delta \leq \pi_i \leq 1/\delta$ $\forall i$. 
\end{theorem}

Next, we consider the singular limit when $\beta \to +\infty$. We state in particular that, for $\beta \to +\infty$, at most a finite number of components of the limit solutions are non-zero.

\begin{theorem}\label{prp sing lim}
Under the same assumptions as in Theorem \ref{prp asymptotic k}, let $\{\mathbf{v}_\beta\}_\beta$ be a family of non negative solutions as above, defined for a sequence of $\beta \to +\infty$. Up to subsequences, there exists $\mathbf{\bar v} =  (\mathbf{\bar w}, \bar u) \in H^1(\Omega)$ such that
\begin{itemize}
	\item the vector $\mathbf{\bar w}$ has at most $\hat N$ non zero components, where $\hat N$ is given by
	\[
		\hat N = C(\Omega) \left(\max_{i}\frac{\lambda k_i - \mu \omega_i}{d_i \mu}\right)^{\frac{n}{2}}
	\]
	and $C(\Omega)$ is a positive constant that only depends on the set $\Omega$;
	\item for all $\alpha \in (0,1)$, we have the estimate
	\[
		\| \bar u \|_{\C^{2,\alpha}(\overline \Omega)} + \max_{i \in \{1, \dots, N\}} \| \bar w_i \|_{Lip(\overline{\Omega})} \leq C_\alpha
	\]
	where $C_\alpha$ is the same constant as in Theorem \ref{prp asymptotic k};
	\item for any $\alpha \in (0,1)$, we also have
	\[
		\lim_{\beta \to +\infty} w_{i,\beta} = \bar w_{i} \quad \text{in } C^{0,\alpha}\cap H^1(\overline{\Omega}), \forall i \qquad \text{and} \qquad \lim_{\beta \to +\infty} u_{\beta} = \bar u \quad \text{in }C^{2,\alpha}(\overline{\Omega}).
	\]
\end{itemize}
\end{theorem}

The next result characterizes the limit solutions and the free-boundary problem that they satisfy.

\begin{theorem}\label{prp free boundary}
We assume the hypotheses of Theorem \ref{prp sing lim}. Any limit $\mathbf{\bar v}$ is such that $\bar w_i \bar w_j \equiv 0$ in $\overline{\Omega}$ for all $i \neq j$. The non zero components of $\mathbf{\bar v}$ are in a finite number and satisfy the following system of complementary inequalities in the sense of measures:
\begin{equation}\label{eqn segr model}
	\begin{dcases}
		- d_i \Delta w_i  \leq (- \omega_i + k_i u -\mu_i) w_i \\
		- \Delta \left( d_i w_i-\sum_{j \neq i} d_j w_j \right) \geq (- \omega_i + k_i u )w_i-\sum_{j \neq i} (- \omega_j + k_j u ) w_j &\text{in $\Omega$}\\
		- D \Delta u = \left(\lambda - \mu u - \sum_{i} k_i w_i \right)u\\
		w_i \partial_\nu w_i = \partial_\nu u = 0 &\text{on $\partial \Omega$}.
	\end{dcases}
\end{equation}
Lastly, if the limit $\mathbf{\bar w}$ has two or more non zero components, then the subset $\mathfrak{N} := \{x \in \overline{\Omega}: w_i = 0,\, \forall i\}$ is a rectifiable set of Hausdorff dimension $n-1$. The set $\mathfrak{N}$ can be written as the disjoint union of two sets, $\mathfrak{R}$ and $\mathfrak{S}$, such that $\mathfrak{R}$ is relatively open and made of the union of a finite number of $\C^{1,\alpha}$ smooth sub-manifolds, while $\mathfrak{S}$ is a set of Hausdorff dimension $n-2$. Moreover, $\mathfrak{R}$ meets orthogonally the boundary $\partial \Omega$ and $\mathfrak{N} \cap \partial \Omega$ is a set of Hausdorff dimension $n-2$, that can be decomposed as the disjoint union of a regular part (finite union of  $\C^{1,\alpha}$ smooth sub-manifolds of codimension $2$) and a singular part. 

\end{theorem}

\begin{remark}
Observe that the Lipschitz bound of the densities is stated for the limit functions. In the case $N$ is a priori bounded, such a bound of the Lipschitz norm also holds for the whole convergent sequence, uniformly in $\beta$. This follows from \cite[Theorem 1.2]{SoaveZilio_ARMA}. It is not clear at the moment whether the same result holds true in the more general case when $N$ is unbounded. We also leave it as an open problem to know if the same a priori estimates holds when the assumption $a_{ij} = a_{ji}$ is removed.
\end{remark}

The next result approximately describes the solution when $N$ is large. Here $\hat N$ stands for the constant in Theorem \ref{prp sing lim}.

\begin{theorem}\label{thm max packs}
We assume the same hypothesis of Theorem \ref{prp free boundary}. There exists $\bar \beta > 0$, independent of $N$, such that if $\beta > \bar \beta$ and $\mathbf{v}_\beta = (\mathbf{w}_{\beta}, u_{\beta})$ is a solution of \eqref{eqn model k} then
\begin{itemize}
	\item either at most $\hat N$ components of $\mathbf{w}_\beta$ are strictly positive and the others are zero;
	\item or the solution is such that
\[
	 \max_{i=1,\dots, N} \|w_{i,\beta}\|_{\C^{0,\alpha}(\Omega)} + \|u_\beta- \lambda/\mu\|_{\C^{2,\alpha}(\Omega)} = o_\beta(1)
\]
for every $\alpha \in (0,1)$.
\end{itemize}
\end{theorem}

In our forthcoming work \cite{BerestyckiZilio_RR}, we actually derive a stronger result, under the additional assumptions that the coefficients in \eqref{eqn model k} do not depend on the density $w_i$. In \cite[Theorem 1.1]{BerestyckiZilio_RR} we show that if $N$ (the number of non zero components of $\mathbf{w}$) is large enough, then \eqref{eqn model k} has only constant solutions, independently of the value of $\beta \geq 0$. The proof hinges on the various a priori estimates that we establish in the present paper.

\noindent\textbf{Structure of paper.} The proofs of Theorems \ref{prp asymptotic k}, \ref{prp sing lim} and \ref{prp free boundary} will be structured in a succession of intermediate results. We will establish successively uniform $L^\infty$ bounds, uniform H\"older bounds, convergence to segregated limits and the upper bound on the number of non-zero components of the limit problem. The structure of the free boundary will follow from already established results, we sketch the arguments in Section 4. We will conclude with a sketch of the proof of Theorem \ref{thm max packs}, as it follows very closely that of \cite[Theorem 4.3]{BerestyckiZilio_PI}.

Before proceeding with the proofs, we point out that the solutions of \eqref{eqn model k} are smooth for $\beta$ and $N$ bounded, and their regularity is only limited by the regularity of the boundary of $\Omega$.

\section{A priori estimates and uniform bounds - Proof of Theorem \ref{prp asymptotic k}} 

First, we prove the uniform regularity estimates for the solutions of \eqref{eqn model k}. We will achieve this in a sequence of partial regularity results.

\begin{proof}[Proof of Theorem \ref{prp asymptotic k} - $L^\infty$ bounds]
We start by showing the uniform upper bound of $\mathbf{v}$. From the equation satisfied by the component $u$ in \eqref{eqn model k}, we find that
\[
  \begin{cases}
	  -D \Delta u = \lambda u - \mu u^2 - \sum_{j\neq i} k_i w_i u \leq \lambda u - \mu u^2 &\text{in $\Omega$}\\
	  \partial_\nu u = 0 &\text{on $\partial \Omega$}
	\end{cases}
\]
Thus, the left hand side is negative if $u > \lambda / \mu$. By the maximum principle, we obtain that $u \leq \lambda / \mu \leq \delta^{-2}$. We now introduce the function $S := D u + \sum_i d_i w_i$. Summing together the $N+1$ equations in the system, we get
\begin{multline*}
	-\Delta S = -\sum_{i} \omega_i w_i - \beta\sum_{j\neq i} a_{ji} w_i w_j + \lambda u - \mu u^2 \leq - \delta^2 D u -\sum_{i} \omega_i w_i + ( \delta^2 D +\lambda) u - \mu u^2 \\
	\leq -\delta^2 S + (\lambda + \delta^2 D) u - \mu u^2 \leq -\delta^2 S + \frac{(\lambda + \delta^2 D )^2}{4 \mu} \leq  -\delta^2 S + \frac{1}{\delta^3}
\end{multline*}
where $\delta >0$ is the constant appearing in \eqref{unifass}. It follows that the left hand side is again negative if $S > 1/\delta^5$. We can conclude that
\[
  0 \leq u \leq \frac{\lambda}{\mu} \leq \frac{1}{\delta^2} \qquad \text{and} \qquad 0 \leq \sum_{i=1}^N w_i \leq \frac{1}{\delta^6}.
\]
This completes the proof of the uniform $L^\infty$ bounds.
\end{proof}

\begin{proof}[Proof of Theorem \ref{prp asymptotic k} - $C^{1,\alpha}$ (and $C^{2,\alpha}$) uniform bounds for $u$]
Once the $L^\infty$ uniform estimate is settled, we can proceed by decoupling the equation in $u$ from those in $\mathbf{w}$. By the previous estimates, $u$ is bounded uniformly in $L^\infty$ and so are all the terms of the equation
\[
	\begin{cases}
		- D \Delta u = \left( \lambda - \mu u - \sum_{i=1}^N k_i w_i\right) u &\text{ in $\Omega$}\\
		\partial_\nu u = 0 &\text{ on $\partial \Omega$}.
	\end{cases}	
\]
By the standard elliptic estimates, we find that $u \in W^{2,p}(\Omega)$ for any $p < +\infty$ and, thus, for any $\alpha \in (0,1)$ there exists $C_\alpha = C(\alpha, \delta, \Omega) > 0$, such that
\[
	\|u\|_{C^{1,\alpha}(\overline{\Omega})} \leq C.
\]
We observe also that, once we will have established the $C^{0,\alpha}$ uniform bounds for $\mathbf{w}$ (more precisely, for the sum of $k_i w_i$), the $C^{2,\alpha}$ uniform bounds for $u$ then follow from Schauder estimates for the above equation. So we focus on the former.
\end{proof}

\begin{proof}[Proof of Theorem \ref{prp asymptotic k} - $C^{0,\alpha}$ uniform bounds for $\mathbf{w}$]
We now turn our attention to the system satisfied by $\mathbf{w}$ in \eqref{eqn model k}. If the number $N$ was a given constant and not arbitrary as in our case, the theory developed in \cite{CaffKarLin, ContiTerraciniVerzini_AdvMat_2005} would be sufficient to show the uniform bounds on the components of $\mathbf{w}_\beta$ independently of $\beta$. Since we want to prove here bounds that are also uniform in $N$, we need to introduce a new method. The remainder of this section is dedicated to the proof of this result.
\end{proof}

We will only need to consider the sub-system satisfied by the components of $\mathbf{w}$. For any $\alpha \in (0,1)$, we wish to show the uniform estimates
\[
  \max_{i \in \{1, \dots, N\}} \| w_i \|_{C^{0,\alpha}(\overline{\Omega})} + \max_{\delta \leq \pi_i \leq 1/	\delta} \left\| \sum_{1=1}^N \pi_i w_i \right\|_{C^{0,\alpha}(\overline{\Omega})} \leq C_\alpha \left\| \sum_{1=1}^N w_i \right\|_{L^{\infty}(\Omega)},
\]
for arbitrarily fixed values of the parameters $\pi_i$, $\pi_i >0$. First of all, we renormalize the components by letting
\[
  \frac{w_i}{\left\|\sum_{i=1}^N w_i \right\|_{L^\infty(\Omega)} } \mapsto w_i \qquad \text{and} \quad \beta \left\|\sum_{i=1}^N w_i \right\|_{L^\infty(\Omega)} \mapsto \beta
\]
By doing so, we end up with the system
\begin{equation}\label{eqn subsys renorm}
	\begin{cases}
		- d_i \Delta w_i  = \left(- \omega_i + k_i u - \beta \sum_{j \neq i} a_{ij} w_j\right) w_i &\text{in $\Omega$}\\
		\sum_{i=1}^N w_i \leq 1 &\text{in $\Omega$}\\
		\partial_\nu w_i = 0 &\text{on $\partial \Omega$}
	\end{cases}
\end{equation}
and the estimate we wish to prove translates into the inequality
\[
  \max_{i \in \{1, \dots, N\}} \| w_i \|_{C^{0,\alpha}(\overline{\Omega})} + \max_{\delta \leq \pi_i \leq 1/	\delta} \left\| \sum_{1=1}^N \pi_i w_i \right\|_{C^{0,\alpha}(\overline{\Omega})} \leq C_\alpha.
\]
To prove it, we assume, by virtue of contradiction, that there exists an exponent $\alpha \in (0,1)$, a sequence of solution $(\mathbf{w}_n, u_n)$ of \eqref{eqn subsys renorm} and $\delta \leq \pi_{i,n} \leq 1/\delta$ such that
\begin{equation}\label{eqn bu assumpA}
	\max_{i=1,\dots,N_n} \|w_{i,n}\|_{C^{0,\alpha}(\overline{\Omega})} + \left\|\sum_{i=1}^{N_n} \pi_{i,n} w_{i,n} \right\|_{C^{0,\alpha}(\overline{\Omega})} \to +\infty
\end{equation}
Observe that the functions $\mathbf{w}_n$, as well as their sum, are uniformly bounded. Therefore, in \eqref{eqn bu assumpA} we have that the unboundedness of the norm is caused solely by the H\"older quotient part.

In order to simplify the exposition of the proof, we will assume from now on that
\[
  \pi_{i,n} = d_i \qquad \forall i,n.
\]
The general case follows in the same spirit by considering the system
\[
	\begin{cases}
		- \frac{d_i}{\pi_{i,n}} \Delta (\pi_{i,n} w_i)  = \left(\frac{- \omega_i}{\pi_{i,n}} + \frac{k_i}{\pi_{i,n}} u - \beta \sum_{j \neq i} \frac{a_{ij}}{\pi_{i,n}\pi_{j,n}} (\pi_{j,n}w_j)\right) (\pi_{i,n}w_{i,n}) &\text{in $\Omega$}\\
		\sum_{i=1}^N (\pi_{i,n}w_i) \leq \frac{1}{\delta} &\text{in $\Omega$}\\
		\partial_\nu (\pi_{i,n}w_i) = 0 &\text{on $\partial \Omega$.}
	\end{cases}
\]
Observe that, by assumption, the coefficients $\pi_{i,n}$ and the ratios $d_i/\pi_{i,n}$, $\omega_i/\pi_{i,n}$, $k_i/\pi_{i,n}$ and $a_{ij}/(\pi_{i,n}\pi_{j,n})$ are bounded from above and away from zero, and the matrix $a_{ij}/(\pi_{i,n}\pi_{j,n})$ is still symmetric. These are the only assumptions that we will need in the proof.

In the following, we will reach a contradiction by using the sequence $(\mathbf{w}_n, u_n)$ and the assumption \eqref{eqn bu assumpA}. To do this, we will need to repeatedly extract sub-sequences in order to refine some properties of the original blow-up sequence. For brevity, by \utss\ we mean ``up to striking out a sub-sequence''.

We introduce the auxiliary sequence of functions $h_n : \overline{\Omega} \to \R_+$, defined as
\[
	h_n(x) := \sum_{i =1}^{N_n} d_i w_{i,n}(x).
\]
Thus, our aim is to reach a contradiction with the assumption
\begin{equation}\label{eqn bu assump}
	\max_{i=1,\dots,N_n} \|w_{i,n}\|_{C^{0,\alpha}(\overline{\Omega})} +  \|h_n\|_{C^{0,\alpha}(\overline{\Omega})} \to +\infty.
\end{equation}

We show the result by means of a blow-up argument. We need to distinguish two cases:

\noindent\textbf{Case W)} There exist a subsequence of $\mathbf{w}_n$ and a constant $C > 0$ such that
\[
	 \|h_n\|_{C^{0,\alpha}(\overline{\Omega})} < C \max_{i=1,\dots,N_n} \|w_{i,n}\|_{C^{0,\alpha}(\overline{\Omega})}.
\]

\noindent\textbf{Case H)} There exist a subsequence of $\mathbf{w}_n$ and a sequence $C_n \to +\infty$ such that
\[
	 \|h_n\|_{C^{0,\alpha}(\overline{\Omega})} > C_n \max_{i=1,\dots,N_n} \|w_{i,n}\|_{C^{0,\alpha}(\overline{\Omega})}.
\]

Observe that both cases are equally possible. Since the number of components in $\mathbf{w}_n$ is not a priori bounded, it may be that the functions are uniformly bounded, but their sum may not be so. The two cases will be addressed separately, but some preliminary results are valid for both, and we present them jointly.

First, as we have already observed, for any fixed $n$, any solution in the sequence $(\mathbf{w}_n, h_n)$ is $C^{2,\alpha}$, in particular $C^1$. Therefore, there exist a sequence of pairs of distinct points $(x_n, y_n) \in \overline{\Omega} \times \overline{\Omega}$, $x_n \neq y_n$, that achieve the H\"older part of the norm of either $\mathbf{w}_n$ or $h_n$, that is such that in \textbf{Case W)} 
\[
	L_n := \max_{i=1,\dots,N_n} \|w_{i,n}\|_{C^{0,\alpha}(\overline{\Omega})} = \max_{i=1,\dots,N_n} \frac{|w_{i,n}(x_n) - w_{i,n}(y_n)|}{|x_n-y_n|^\alpha} \quad \text{ and } L_n \geq \frac{1}{C} \|h_n\|_{C^{0,\alpha}(\overline{\Omega})},
\]
while in \textbf{Case H)}
\[
	L_n := \|h_n\|_{C^{0,\alpha}(\overline{\Omega})} =  \frac{|h_{n}(x_n) - h_{n}(y_n)|}{|x_n-y_n|^\alpha} \quad \text{ and } L_n \geq C_n \max_{i=1,\dots,N_n} \|w_{i,n}\|_{C^{0,\alpha}(\overline{\Omega})}.
\]
We also define the sequence $r_n := |x_n - y_n|$. Notice the different definitions of $L_n$ in the two cases. Either way, we have $L_n \to +\infty$ (in view of \eqref{eqn bu assump}) and
\[
	L_n \leq \max_{i=1,\dots,N_n} \|w_{i,n}\|_{C^{0,\alpha}(\overline{\Omega})} +  \|h_n\|_{C^{0,\alpha}(\overline{\Omega})}  \leq C' L_n
\]
for some constant $C' > 1$.

We now introduce the blow-up sequence at the core of the contradiction argument. For any $n \in \N$, we let
\begin{equation}\label{eqn bu seq}
	\mathbf{W}_n(x) := \frac{1}{L_n r_n^\alpha} \mathbf{w}_n(x_n +r_n x) \quad \text{and} \quad H_n(x) := \frac{1}{L_n r_n^\alpha} h_n(x_n +r_n x).
\end{equation}
The functions $(\mathbf{W}_n, H_n)$ are defined on the sets $\Omega_n := \frac{1}{r_n}(\Omega-x_n)$. Since the functions in the sequence $(\mathbf{w}_n, h_n)$ are uniformly bounded, we see that necessarily $r_n \to 0^+$ as $n \to +\infty$. As a result, depending on the behavior of the sequences $(x_n, y_n)$, we have that $\Omega_n \to \Omega_\infty$ \utss{}, where $\Omega_\infty$ is either the entire space $\R^n$ or a half space. Here, when we write $\Omega_n \to \Omega_\infty$ we mean that:
\begin{itemize}
	\item the sets $\Omega_n$ are uniformly regular in $n$;
	\item for any compact set $K \subset \Omega_\infty$ (or $K \subset \R^n \setminus \overline{\Omega_\infty}$) there exists $\bar n$ such that $K \subset \Omega_n$  ($K \subset \R^n \setminus \overline{\Omega_n }$ respectively) for any $n \geq \bar n$.
\end{itemize}
Observe that by definition, for any $x, y \in \overline{\Omega_n}$
\begin{equation}\label{eqn unif Holder bu}
	\max_{i=1,\dots,N_n} |W_{i,n}(x)-W_{i,n}(y)| + |H_{n}(x)-H_{n}(y)| \leq C'' |x-y|^\alpha
\end{equation}
for a constant $C''>0$ that does not depend on $n$.

We derive the equations satisfied by the blow-up sequence $(\mathbf{W}_n, H_n)$ by scaling \eqref{eqn subsys renorm} accordingly. We find 
\begin{equation}\label{eqn bu}
	\begin{cases}
		-d_i \Delta W_{i,n} = \left( \aa_{i,n} - M_n \sum_{j\neq i} a_{ij} W_{j,n} \right)W_{i,n} &\text{ in $\Omega_n$}\\
		-\Delta \left(\sum_{i=1}^{N_n} d_{i} W_{i,n}(x)\right) = \sum_{i=1}^{N_n} \aa_{i,n} W_{i,n}(x) - M_n \sum_{i=1}^{M_n} W_{i,n} \sum_{ j \neq i}a_{ij} W_{j,n} &\text{ in $\Omega_n$}\\
		\partial_\nu W_{i,n} = \partial_\nu H_n = 0 &\text{ on $\partial \Omega_n$.}
	\end{cases}
\end{equation}
Here we have defined
\begin{equation}\label{eqn a and M}
  \aa_{i,n} = (-\omega_i + k_i u_n(x_n +r_n \cdot )) r_n^2 \quad \text{and} \quad M_n = \beta_n L_n r_n^{2+\alpha}.
\end{equation}
The uniform $L^\infty$ bound of $u_n$ implies that the sequence $\aa_{i,n}$ converges uniformly towards 0,
\[
  \|\aa_{i,n}\|_{L^\infty(\Omega_n)} \leq \delta^{-3} r_n^2 \to 0.
\]
Moreover, by definition we have
\[
  \left\|\sum_{i=1}^{N_n}\aa_{i,n}  W_{i,n} \right\|_{L^\infty(\Omega_n)} \leq \delta^{-3} \left\|\sum_{i=1}^{N_n}r_n^2  \frac{1}{L_nr_n^{\alpha}} w_{i,n} \right\|_{L^\infty(\Omega)}  \leq C(\delta) \frac{r_n^{2-\alpha}}{L_n} \to 0
\]
However, we have no information a priori on the possible behavior of the sequence of positive numbers $M_n$.

We now analyze \textbf{Case W)} and \textbf{Case H)} separately.

\noindent\textbf{Case W)}
First we assume that, up to a relabelling, the function in $\mathbf{w}_n$ are ordered decreasingly with respect to their H\"older seminorms. Thus we have 
\[
	\max_{i=1,\dots,N_n} \|w_{i,n}\|_{C^{0,\alpha}(\overline{\Omega})} = \|w_{1,n}\|_{C^{0,\alpha}(\overline{\Omega})}
\]
In this case, there exists a sequence $z_n \in \partial B_1 \cap \overline{\Omega_n}$ such that
\begin{equation}\label{eqn unit osc W}
	1 =  {|W_{1,n}(0)-W_{1,n}(z_n)|}= \max_{i=1,\dots,N_n} \frac{|W_{i,n}(x)-W_{i,n}(y)|}{|x -y|^{\alpha}} \qquad \text{for all $x \neq y \in \overline{\Omega_n}$}.
\end{equation}
In particular, the functions in $(\mathbf{W}_n, H_n)$ have uniformly bounded H\"older seminorm. We wish to show that they are also bounded in $x = 0$. This will imply local uniform convergence to a vector (of possibly infinitely many components) $(\mathbf{\bar W}, \bar H)$, as we show in Lemma \ref{lem lim bu WH}.

The next result is a simple adaptation of \cite[Lemma 6.9]{TVZ1}.

\begin{lemma}\label{lem estim via energy}
For any fixed $R > 0$ there exists $C = C(R) > 0$ such that for any $n \in \N$ and $i = 1, \dots, N_n$ we have
\[
    M_n \int_{\partial B_R \cap \Omega_n} W_{i,n}^2 \sum_{j \neq i} a_{ij} W_{j,n} \leq C(R) (1 + W_{i,n}(0)).
\]
\end{lemma}
\begin{proof}
In system \eqref{eqn bu}, we multiply the equation in $W_{i,n}$ by $W_{i,n}$ itself. Integrating by parts on the ball $B_R$, we find
\[
  D(R) := \frac{1}{R^{N-2} } \left( \int_{B_R} d_i |\nabla W_{i,n}|^2 + \left( -\aa_{i,n} + M_n \sum_{j \neq i}a_{ij} W_{j,n}\right) W_{i,n}^2 \right) = \frac{1}{r^{N-2}} \int_{\partial B_R} W_{i,n} \partial_{\nu} W_{i,n}
\] 
We now let
\[
  E(R) := \frac{1}{r^{N-1}} \int_{\partial B_R} W_{i,n}^2.
\]
Since the functions involved are regular, we have
\[
    E(2R) - E(R) =  \int_{R}^{2R} E'(r) = \int_{R}^{2R} \frac{2}{r} D(r).
\]
The statement will follow once we have suitably estimated the two sides of the previous identity. We start from the left hand side. By the uniform bounds on the H\"older seminorm of $\mathbf{W}_n$, we have
\begin{multline*}
    E(2R) - E(R) = \int_{\partial B} \left( W_{i,n}^2(2Rx) - W_{i,n}^2(Rx) \right) = \\
    \int_{\partial B} \left( W_{i,n}(2Rx) - W_{i,n}(Rx) \right)  \left( W_{i,n}(2Rx) + W_{i,n}(Rx) \right) 
    \leq C(R)(|W_{i,n}(0)| +1).
\end{multline*}
For what concerns the right hand side, we have directly
\[\begin{split}
    \int_{R}^{2R} \frac{2}{r} D(r) &\geq \min_{s \in [R,2R]} D(s) \geq \frac{1}{R^{N-2}} \left( \frac{M_n}{2^{N-1}} \int_{\partial B_R} W_{i,n}^2 \sum_{j \neq i}a_{ij} W_{j,n} - \int_{\partial B_{2R}}  |\aa_{i,n}| W_{i,n} \right)\\
    &\geq C \left(M_n \int_{\partial B_R} W_{i,n}^2 \sum_{j \neq i}a_{ij} W_{j,n}  - \|\aa_{i,n}\|_{L^{\infty}} (|W_{i,n}(0)| +1)\right).\qedhere
\end{split}
\]
\end{proof}

The uniform estimate in Lemma \ref{lem estim via energy} is key in order to prove the local boundedness of the blow-up sequence $(\mathbf{W}_n, H_n)$.
 
\begin{lemma}\label{Hn bounded in W}
The sequence of functions $(\mathbf{W}_n, H_n)$ is bounded locally uniformly.
\end{lemma}
\begin{proof}
We start by observing that it is sufficient to prove that $H_n$ is bounded in $x = 0$. Indeed, we have
\[
  H_n(0) = \sum_{i=1}^{N_n} d_i W_{i,n}(0) \geq \delta \sum_{i=1}^{N_n} W_{i,n}(0)
\]
and the functions in $(\mathbf{W}_n, H_n)$ satisfy by construction the uniform estimate \eqref{eqn unif Holder bu}. Let us then assume, by contradiction, that $H_n(0) \to + \infty$. We distinguish two cases.

\noindent\textbf{Case 1)} $H_n(0) \to +\infty$ but there exists $C > 0$ such that $W_{i,n}(0) \leq C$ for all $n \in \N$ and $i = 1, \dots, N_n$.  This evidently implies that $N_n\to +\infty$. Let $R> 0$ be a constant that will be chosen later in the proof. Observe that, by assumption, we have
\[
	\frac{\delta}{2} H_n(0) \leq \sum_{j \neq 1} a_{ij} W_{j,n}(0) \leq \frac{1}{\delta} H_n(0) \qquad \text{for $n$ sufficiently large.}
\]
By Lemma \ref{lem estim via energy} and the uniform bounds in \eqref{eqn unif Holder bu}, we find
\begin{equation}\label{eqn estim 4}
	\frac{\delta}{4}M_n H_n(0) \int_{\partial B_R\cap \Omega_n} W_{1,n}^2 \leq \frac{\delta}{2} M_n (H_n(0) - C R^\alpha) \int_{\partial B_R\cap \Omega_n} W_{1,n}^2 \leq C(R) (1 + W_{1,n}(0)).
\end{equation}
Since, in the present case, $W_{1,n}$ is bounded and has bounded $C^{0,\alpha}$ seminorm, by \eqref{eqn unit osc W} we find that it converges \utss\ to a non zero globally H\"older continuous function. As a result, there exist $\bar R > 0$ and $\bar C > 0$ such that
\[
	\bar C < \int_{\partial B_{\bar R} \cap \Omega_n} W_{1,n}^2 \leq 2 \bar C
\]
for $n$ sufficiently large. Plugging this information back into \eqref{eqn estim 4}, we find that 
\[
	M_n H_n(0) \leq C \qquad \text{for all $n \in \N$}.
\]
It follows that $M_n \to 0$. By the H\"older uniform bounds \eqref{eqn unif Holder bu}, there exists a constant $\Lambda \geq 0$ such that for any compact set $K \subset \R^n$, \utss, $M_n H_n \to \Lambda$ uniformly in $K$. From \eqref{eqn bu}, we consider the equation satisfied  by $W_{1,n}$. By the current assumptions, the sequence $W_{1,n}$ converges locally in $C^{0,\alpha'}$ for any $\alpha' < \alpha$ to a function $\bar W$ that is defined on $\Omega_\infty$, which is non negative, non constant, globally $\alpha$-H\"older continuous, and a solution to 
\[
	\begin{cases}
		-\Delta \bar W = - \Lambda \bar W &\text{in $\Omega_\infty$}\\
		\partial_\nu \bar W = 0  &\text{on $\partial \Omega_\infty$}.
	\end{cases}
\]
If $\partial \Omega_\infty$ is not empty, we can moreover evenly extend $\bar W$ to the whole $\R^n$ to a thus positive, non constant global solution of $-\Delta u = - \Lambda u$. Either way, we find a contradiction.

\noindent\textbf{Case 2)} $H_n(0) \to +\infty$ and there exists $i$ such that, up to a subsequence, $W_{i,n}(0) \to +\infty$. For any $n$, let $i_n \in \{1, \dots, N_n\}$ be the index such that
\[
	W_{i_n,n}(0) = \max_{i = 1, \dots, N_n} W_{i,n}(0).
\]
We can argue similarly as in the first case, and find from Lemma \ref{lem estim via energy} the estimate
\[
	M_n H_n(0) W_{i_n,n}(0)^2 \leq C'(R) M_n \int_{\partial B_R} H_n W_{i_n,n}^2 \leq C''(R) (1 + W_{i_n,n}(0)).
\]
From this we conclude that there exists $C > 0$ such that
\[
	M_n H_n(0) \max_{i = 1, \dots, N_n} W_{i,n}(0) \leq C.
\]
Once again, this implies that $M_n \to 0$. By exploiting the uniform H\"older bounds \eqref{eqn unif Holder bu}, we find that there exists a constant $\Lambda \geq 0$ such that for any compact set $K \subset \R^n$, \utss\, $M_n H_n W_{1,n} \to \Lambda$ uniformly in $K$. We consider the sequence of function $w_n := W_{1,n} - W_{1,n}(0)$. By virtue of the same reasoning as before, we can use its local H\"older limit to build a non constant global solution of $-\Delta u = - \Lambda$. We find again a contradiction.
\end{proof}

As a result of Lemma \ref{Hn bounded in W}, the sequence $(\mathbf{W}_n, H_n)$ is uniformly bounded in $x=0$ and by \eqref{eqn unif Holder bu} it has uniformly bounded $\alpha$-H\"older seminorm. We now show that there exists a vector $(\mathbf{\bar W}, \bar H)$ of possibly infinitely many components, that is the limit of a subsequence of $(\mathbf{W}_n, H_n)$.

\begin{lemma}\label{lem lim bu WH}
There exist a subsequence of $(\mathbf{W}_n, H_n)$, a vector $(\mathbf{\bar W}, \bar H)$ of possibly infinitely many components and a constant $C >0$ such that
\begin{itemize}
  \item $\bar W_i \geq 0$ for all $i = 1, \dots$ and $\bar H \geq 0$;
  \item $\max_{i = 1,\dots} \bar W_i(0) + \bar H(0) \leq C$;
  \item $\max_{i = 1,\dots} |\bar W_i(x) - \bar W_i(y)| + |\bar H(x) - \bar H(y)| \leq C|x-y|^\alpha$ for all $x,y \in \Omega_\infty$;
  \item for all compact set $K \subset \Omega_\infty$, $H_n \to \bar H$ and, for all $i$, $W_{i,n} \to \bar W_i$ in $C^{0,\alpha'}(K)$ for all $\alpha' \in (0,\alpha)$;
  \item lastly, there exists $z \in \partial B_1 \cap \overline{\Omega_\infty}$ such that $|\bar W_1(0) - \bar W_1(x)| = 1$. 
\end{itemize}
\end{lemma}
\begin{proof}
The proof follows by the classical Ascoli-Arzel\`a compactness criterion and a diagonal extraction argument.
\end{proof}

We can now use the existence of a limit function $(\mathbf{\bar W}, \bar H)$ in order to pass to the limit in \eqref{eqn bu}. We differentiate among three possible behaviors of the sequence $M_n$, which are addressed by the three following lemmas. 
\begin{lemma}
The sequence $M_n$ is bounded away from 0. That is, there $C > 0$ such that $M_n > C > 0$.
\end{lemma}
\begin{proof}
Assume that, \utss\, if holds $M_n \to 0$. In this case, we can pass directly to the limit in the equation in \eqref{eqn bu} satisfied by $\bar W_{1}$, and find
\[
	\begin{cases}
		-\Delta \bar W_1 = 0 &\text{in $\Omega_\infty$}\\
		\partial_\nu \bar W_1 = 0  &\text{on $\partial \Omega_\infty$}.
	\end{cases}
\]
As before, upon an eventual even extension, we can build a positive harmonic function, which is globally $\alpha$-Holder continuous and non constant. A contradiction.
\end{proof}

In order to prove the next two results, which will enable us to reach the final contradiction in \textbf{Case W)}, we need to introduce an auxiliary sequence of functions. We let
\[
  H_{1,n} := \sum_{i=2}^{N_n} d_i W_{i,n} = H_n - d_1 W_{1,n}.
\]
It is easy to show that this sequence converges to its limit $\bar H_1 = \bar H - d_1 \bar W_{1} \geq 0$. We observe that the pairs $(W_{1,n}, H_{1,n})$ are solutions of the system of inequalities
\begin{equation}\label{eqn W1h1}
	\begin{cases}
		-\Delta (d_1 W_{1,n}) \leq \left( \aa_{1,n} - M_n \delta^2 H_{1,n} \right)W_{1,n}\\
		-\Delta H_{1,n} \leq \left(\delta^{-4} r_n^2 - M_n \delta^2 W_{1,n}\right) H_{1,n} &\text{in $\Omega_n$}\\
		-\Delta (d_1 W_{1,n} - H_{1,n}) \geq \aa_{1,n} W_{1,n} - r_n^2 \delta^{-4} H_{1,n}\\
		\partial_\nu W_{1,n} = \partial_\nu H_{1,n} = 0 &\text{on $\partial \Omega_n$.}
	\end{cases}
\end{equation}

\begin{lemma}
The sequence $M_n$ is unbounded. That is, $M_n \to +\infty$.
\end{lemma}
\begin{proof}
We argue again by contradiction. Without loss of generality, we can assume that, \utss\, $M_n \to 1$. Passing to the limit in \eqref{eqn W1h1} we find
\[
	\begin{cases}
		-\Delta (d_1 \bar W_{1}) \leq - \delta^2 \bar H_{1} \bar W_{1}\\
		-\Delta \bar H_{1} \leq - \delta^2 \bar W_{1} \bar H_{1} &\text{in $\Omega_\infty$}\\
		-\Delta (d_1 \bar W_{1} - \bar H_{1}) \geq 0\\
		\partial_\nu \bar W_{1} = \partial_\nu \bar H_1 = 0 &\text{on $\partial \Omega_\infty$.}
	\end{cases}
\]
As a result of \cite[Lemma A.3]{STTZ}, we find that necessarily $\bar W_1 \equiv 0$ or $\bar H_1 \equiv 0$. Since $\bar W_1$ is non constant, it must be that $\bar H_1 \equiv 0$. Thus, again by the previous system, we conclude that $\bar W_1$ is a non constant positive harmonic function. A contradiction.
\end{proof}

We exclude the last possibility, $M_n \to +\infty$, and thus conclude the proof of \textbf{Case W)}.
\begin{lemma}
It cannot be that $M_n \to +\infty$.
\end{lemma}
\begin{proof}
We can reasoning similarly as before by contradiction. First, let $\eta \in C_0^\infty(\R^n)$ be any non negative test function. By testing the equation in $d_1 W_{1,n}$ in \eqref{eqn W1h1}, we find that
\[
   M_n \delta^2 \int H_{1,n} W_{1,n} \eta \leq \int \aa_{1,n}W_{1,n} +  (d_1 W_{1,n}) \Delta\eta.
\]
Since the right hand side is uniformly bounded and $M_n \to +\infty$, by passing to the limit we obtain
\[
  \int \bar H_{1} \bar W_{1} \eta = 0 \qquad \text{for all $\eta \in C_0^\infty(\R^n)$, $\eta \geq 0$}.
\]
We then conclude that $H_{1} \bar W_{1} = 0$ in $\R^n$. Moreover 
\[
	\begin{cases}
		-\Delta (d_1 \bar W_{1}) \leq 0\\
		-\Delta \bar H_{1} \leq 0 &\text{in $\Omega_\infty$}\\
		-\Delta (d_1 \bar W_{1} - \bar H_{1}) \geq 0\\
		\partial_\nu \bar W_{1} = \partial_\nu \bar H_1 = 0 &\text{on $\partial \Omega_\infty$}
	\end{cases}
\]
where the last condition is meant in a weak sense. Thus, by \cite[Lemma A.4]{STTZ}, we find again that one between $\bar W_1$ and $\bar H_1$ has to be identically $0$. Since $\bar W_1$ is non constant, it must be that $\bar H_1 \equiv 0$. Once again, we have that $\bar W_1$ is a non constant positive harmonic function. A contradiction.
\end{proof}
With this final result we have reached a contradiction and concluded the proof in \textbf{Case W)}.

\noindent\textbf{Case H)} 
In this case there exists a sequence $z_n \in \partial B_1 \cap \overline{\Omega_\infty}$ such that
\begin{equation}\label{eqn unit osc H}
	1 =  {|H_{n}(0)-H_{n}(z)|} \geq C_n \max_{i=1,\dots,N_n} \frac{|W_{i,n}(x)-W_{i,n}(y)|}{|x -y|^{\alpha}} \qquad \text{for all $x \neq y$}.
\end{equation}
Recall that, by assumption, $C_n \to +\infty$. In particular, the sequence $H_n$ has uniformly bounded H\"older seminorm and all $\mathbf{W}_n$ have vanishing H\"older seminorm. More precisely, we already have that
\begin{equation}\label{eqn ultimate bound W}
	\max_{i = 1, \dots, N_n} |W_{i,n}(x)-W_{i,n}(y)| \leq \frac{1}{C_n}|x-y|^{\alpha}	 \qquad \text{for all $x,y \in \Omega_n$.}
\end{equation}
Up to a relabeling, we now assume that, the functions in $\mathbf{W}_n$ are order in such a way that
\[
	d_1 W_{1,n}(0) \geq d_i W_{i,n}(0) \qquad \text{for all $i = 1, \dots, N_n$}.
\]

 As in the previous case, we wish to show that the sequence $(\mathbf{W}_n, H_n)$ is bounded in $x = 0$. We first derive two differential inequalities that are satisfied by the function $H_n$.

\begin{lemma}\label{lem diff ineq Hn}
There exists a numerical sequence $\eps_n \to 0$ such that
\[
	\begin{cases}
		-\Delta H_n \leq \eps_n H_n - M_n \delta^3 \left(H_n^2 -  \sum_{i=1}^{N_n} d_i^2 W_{i,n}^2\right) &\text{in $\Omega_n$}\\
		-\Delta H_n \geq -\eps_n H_n - M_n \delta^{-3} \left(H_n^2 - \sum_{i=1}^{N_n} d_i^2 W_{i,n}^2\right) &\text{in $\Omega_n$}\\
		\partial_\nu H_n = 0 &\text{on $\partial \Omega_n$.}
	\end{cases}
\]
\end{lemma}
\begin{proof}
We prove the first inequality. The second one follows by similar reasoning. We start from the equation \eqref{eqn bu} satisfied by $W_{i,n}$ and find
\begin{multline*}
	-\Delta (d_i W_{i,n}) = \aa_{i,n} W_{i,n} - M_n W_{i,n} \sum_{j \neq i} a_{ij} W_{j,n} \\
	\leq \eps_{n} d_i W_{i,n} - M_n \delta^2 W_{i,n} \left(\sum_{j\neq i} d_{j} W_{j,n} + d_i W_{i,n} - d_i W_{i,n}\right) \\
	\leq \eps_{n} d_i W_{i,n} - M_n \delta^3  \left(H_n d_i W_{i,n} - d_i ^2 W_{i,n}^2\right)
\end{multline*}
where the numerical sequence $\eps_{n} := \frac{1}{d_i} \sup_{x\in \Omega_n} \max_{i=1,\dots,N_n} \aa_{i,n}$ converges, by \eqref{eqn a and M}, to $0$. Summing the inequalities for $i = 1, \dots, N_n$, we find
\[
	-\Delta H_n \leq \eps_{n}H_n - M_n \delta^3 \left(H_n^2 -  \sum_{i=1}^{N_n} d_i^2 W_{i,n}^2\right). \qedhere
\]
\end{proof}

We point out that $H_n^2 \geq \sum_{i=1}^{N_n} d_i^2 W_{i,n}^2$. Moreover, by \eqref{eqn a and M} and the preliminary discussion there, we have that
\[
  \left\|\eps_n H_{n} \right\|_{L^\infty(\Omega_n)} \leq C(\delta) \frac{r_n^{2-\alpha}}{L_n} \to 0.
\]
By virtue of the same reasoning of Lemma \ref{lem estim via energy}, we find
\begin{lemma}\label{lem estim via energy H}
For any fixed $R > 0$ there exists $C > 0$
\[
    M_n \delta^3 \int_{\partial B_R \cap \Omega_n} H_{n} \left(H_n^2 -  \sum_{i=1}^{N_n} d_i^2 W_{i,n}^2\right) \leq C(R) (1 + H_{n}(0)).
\]
\end{lemma}

As before, we can use the previous result to prove the boundedness of the blow-up sequence.
\begin{lemma}\label{Hn bounded in H}
The sequence $(\mathbf{W}_n, H_n)$ is bounded locally uniformly.
\end{lemma}
\begin{proof}
We assume, by contradiction, that $H_n(0) \to +\infty$. We show first that $M_n H_n(0)$ is bounded uniformly. We consider two distinct cases:

\noindent\textbf{Case A)} Assume first that there exists $\eps > 0$ such that, for $n$ large enough
\[
	d_1 W_{1,n}(0) < (1-\eps) H_n(0).
\]
Then, recalling \eqref{eqn ultimate bound W}, we have
\begin{multline*}
	 \sum_{i=1}^{N_n} d_i^2 W_{i,n}^2(x) \leq \sum_{i=1}^{N_n} \left( d_i W_{i,n}(0) + \frac{d_i}{C_n}|x|^\alpha\right) d_i W_{i,n}(x)  \leq \left( d_1 W_{1,n}(0) + \frac{1}{\delta C_n}|x|^\alpha\right) \sum_{i=1}^{N_n}  d_i W_{i,n}(x)\\
	 \leq \left((1-\eps) H_{n}(0) + \frac{1}{\delta C_n}|x|^{\alpha} \right) H_n(x) \leq \left((1-\eps) H_{n}(x) + (1-\eps) |x|^{\alpha} + \frac{1}{\delta C_n}|x|^{\alpha} \right) H_n(x).
\end{multline*}
By Lemma \ref{lem estim via energy H}, we have
\[
	M_n \delta^3 \int_{B_R} \left(\eps H_n  - (1-\eps) |x|^{\alpha} - \frac{1}{\delta C_n}|x|^{\alpha} \right) H_n^2 \leq C(R) (1+H_n(0)).
\]
Recalling that $H_n(0) \to +\infty$ and that $H_n$ has uniformly bounded H\"older seminorm, we find that
\[
	    M_n \delta^3 \frac{\eps}{2}  H_n(0)^3 \leq M_n \delta^3 C'(R)\int_{\partial B_R \cap \Omega_n}  \frac{\eps}{2} H_{n}^3 \leq C''(R) (1 + H_{n}(0)).
\]
Thus $M_n H_n(0)^2$ is bounded uniformly.

\noindent\textbf{Case B)} Assume now that
\[
	d_1 W_{1,n}(0) \to H_n(0).
\]
We are in a similar situation as the one in \textbf{Case 2)} of Lemma \ref{Hn bounded in W}. By considering the equation satisfied by $W_{1,n}(0)$, we find from Lemma \ref{lem estim via energy} the estimate
\[
	\frac{1}{2} M_n H_n(0)^3 \leq C'(R) M_n \int_{\partial B_R} H_n W_{1,n}^2 \leq C''(R) (1 + W_{1,n}(0)) \leq 2C''(R) (1+H_n(0))
\]
and again, this implies that $M_n H_n(0)^2$ is bounded uniformly.

In either cases, there exists $C \geq 0$ such that $M_n H_n(0)^2 \leq C$. In particular we have $M_n \to 0$. Up to striking out a subsequence, we obtain that
\[
	M_n H_n^2 \to C \geq 0 \quad \text{and} \quad M_n H_n \to 0 \qquad \text{locally uniformly in $\Omega_n$}.
\]
Let $u \in C^{0,\alpha}(\Omega_\infty)$ be the local uniform limit of $H_n - H_n(0)$. Observe that, again by assumptions, $u$ is non constant. Let $x,y \in \Omega_\infty$. We have
\begin{multline*}
  M_n \left( H_n^2(x) - \sum_{i=1}^{N_n} d_i W_{i,n}^2(x)\right) - M_n \left( H_n^2(y) - \sum_{i=1}^{N_n} d_i W_{i,n}^2(y)\right) \\
  = M_n\left(H_n^2(x) - H_n^2(y)\right)  - M_n \left(\sum_{i=1}^{N_n} d_i W_{i,n}^2(x) - \sum_{i=1}^{N_n} d_i W_{i,n}^2(y)\right)\\
  = M_n \left(H_n(x) + H_n(y)\right)\left(H_n(x) - H_n(y)\right)  \\
  - M_n \left(\sum_{i=1}^{N_n} d_i^2 (W_{i,n}(x) + W_{i,n}(y)) (W_{i,n}(x) - W_{i,n}(y))\right).
\end{multline*}
Since $M_n \to 0$, \utss{}, we find that there exists $\Lambda \geq 0$ such that, locally uniformly in $\Omega_\infty$,
\[
  M_n \left( H_n^2(x) - \sum_{i=1}^{N_n} d_i W_{i,n}^2(x)\right) \to \Lambda \geq 0.
\]
By passing to the limit the inequalities in Lemma \ref{lem diff ineq Hn}, up to an even extension of the function $u$ to the whole $\R^n$, we find that
\[
	\begin{cases}
		-\Delta u \leq - \delta^{3} \Lambda &\text{in $\R^n$}\\
		-\Delta u \geq- \delta^{-3} \Lambda &\text{in $\R^n$}
	\end{cases}
\]
for a non negative constant $\Lambda$. We have two possibilities.
\begin{enumerate}
\item $\Lambda = 0$. In this case, the limit function $u$ is harmonic, globally H\"older continuous and non constant, a contradiction.
\item $\Lambda > 0$. Let $\Gamma(x) = \Gamma(|x|)$ be the fundamental solution of the Laplacian centered at $0$. Multiplying the first inequality satisfied by $u$ with $\Gamma(x) - \Gamma(R)$ and integrating by parts in the ball $B_R(0)$, we find
\[
	- \frac{1}{|\partial B_R|} \int_{\partial B_R} u(x) = \frac{1}{|\partial B_R|} \int_{\partial B_R} (u(0) - u(x)) \leq - \delta^{3} \Lambda \int_{B_R} (\Gamma(x) - \Gamma(R)) \leq -C' R^2
\]
where $C' > 0$. On the other hand, exploiting once more the H\"older continuity of $u$, we conclude
\[
	\left|\frac{1}{|\partial B_R|} \int_{\partial B_R} u(x) \right| \leq C R^{\alpha}
\]
and the two assertions are in contradiction.
\end{enumerate}
Thus we conclude that $H_n$ is bounded. This also implies the boundedness of $\mathbf{W}_n$.
\end{proof}

Since $H_n$ is bounded, similarly to \textbf{Case W)}, we can pass to the limit, up to a possible subsequence. Let $\bar H$ be the uniform local limit of $H_n$. If needed, we assume implicitly that $\bar H$ has been extended evenly to a function defined over the whole space $\R^n$. We recall that, necessarily, $\bar H$ is globally H\"older continuous, non-negative and non constant.

\begin{lemma}
There exists $C > 0$ such that $M_n \geq C$.
\end{lemma}
\begin{proof}
If $M_n \to 0$ along a subsquence, from Lemma \ref{lem diff ineq Hn} we deduce that up to extracting a converging subsequence $\bar H$ is also harmonic, a contradiction.
\end{proof}

\begin{lemma}
There exists $C > 0$ such that $M_n \leq C$.
\end{lemma}
\begin{proof}
We assume that $M_n \to +\infty$ along a subsequence. For $R > 0$ and $x_0 \in \R^n$, let us consider the function 
\[
	\eta(x) := \begin{cases} 
		1 &\text{in $B_R(x_0)$}\\
		R+1-|x| &\text{in $B_{R+1}(x_0)\setminus B_R(x_0)$}\\
		0 &\text{in $\R^n \setminus B_{R+1}(x_0)$}.
	\end{cases}
\]
We multiply the first inequality in Lemma \ref{lem diff ineq Hn} by $H_n \eta^2$ and integrate by parts. This yields
\[
	\int |\nabla (H_n \eta)|^2 + M_n \delta^3 \left(H_n^2-\sum_{i=1}^{N_n} d_i^2 W_{i,n}\right) H_n \eta^2 \leq \int |\nabla \eta|^2 H_n^2.
\]
By the definition of $\eta$, we have
\[
	\frac{1}{4} M_n \delta^3 \int_{B_R(x_0)} \left(H_n^2-\sum_{i=1}^{N_n} d_i^2 W_{i,n}^2\right) H_n \leq  \int_{B_{R+1}(x_0) \setminus B_R(x_0)} H_n^2 \leq C \left(1 + R^{N-1+2\alpha} \right).
\]
We recall that the function in the integral in the left hand side is non negative. We adopt a similar reasoning as that of Lemma \ref{Hn bounded in H}. We consider two distinct cases:

\noindent\textbf{Case A)} Assume first that $\eps > 0$ such that, for $n$ large enough
\[
	d_1 W_{1,n}(0) < (1-\eps) H_n(0).
\]
Then, by \eqref{eqn ultimate bound W}, we find
\begin{multline*}
	 \sum_{i=1}^{N_n} d_i^2 W_{i,n}^2(x) \leq \sum_{i=1}^{N_n} \left( d_i W_{i,n}(0) + \frac{d_i}{C_n}|x|^\alpha\right) d_i W_{i,n}(x)  \leq \left( d_1 W_{1,n}(0) + \frac{d_1}{C_n}|x|^\alpha\right) \sum_{i=1}^{N_n}  d_i W_{i,n}(x)\\
	  \leq \left( d_1 W_{1,n}(x_0) + \frac{d_1}{C_n}|x-x_0|^\alpha +  \frac{d_1}{C_n}|x|^\alpha\right) \sum_{i=1}^{N_n}  d_i W_{i,n}(x) \\
	 \leq \left((1-\eps) H_{n}(x_0) +  \frac{d_1}{C_n}|x-x_0|^\alpha +  \frac{d_i}{C_n}|x|^{\alpha} \right) H_n(x) \\
	 \leq \left((1-\eps) H_{n}(x) + (1-\eps) |x-x_0|^{\alpha} +  \frac{d_1}{C_n}|x-x_0|^\alpha +  \frac{d_1}{C_n}|x|^{\alpha} \right) H_n(x).
\end{multline*}
But then we have, for any $x_0 \in \R^n$
\[
	\frac{1}{4} M_n \delta^3 \int_{B_R(x_0)} \left(\eps H_n  - (1-\eps) |x-x_0|^{\alpha} -  \frac{d_1}{C_n}|x-x_0|^\alpha - \frac{d_1}{C_n}|x|^{\alpha} \right) H_n^2 \leq C \left(1+ R^{N-1+2\alpha}\right).
\]
Since by construction there exists $z_n \in \partial B$ such that $|H_n(0) - H_n(z_n)| = 1$, we find that there also exists a bounded sequence $y_n \in \overline{B_1}$ such that $H_n(y_n) > \frac12$. By taking $R> 0$ sufficiently small, $x_0 = y_n$ and $n$ large, we conclude from the last inequality that $M_n$ must be bounded from above.

\noindent\textbf{Case B)} Assume now that
\[
	d_1 W_{1,n}(0) \to H_n(0).
\]
Up to striking out a subsequence, we observe that both $W_{1,n}$ and $H_n$ converge locally uniformly on $\R^n$ to their respective limits $\bar W_1$ and $\bar H$. Moreover $d_1 W_{1,n} \leq H_n$ in $\Omega_n$. Since the local uniform limit of $d_1 \bar W_{1}$ is a constant, while $\bar H$ is necessarily not, we find that in any $B_R$
\[
	\bar H \geq d_1 \bar W_{1} \geq 0 \qquad \text{and} \qquad \bar H \not \equiv d_1 \bar W_{1}.
\]
Similarly to before, we find
\[
	 \sum_{i=1}^{N_n} d_i^2 W_{i,n}^2(x) \leq \sum_{i=1}^{N_n} \left( d_i W_{i,n}(0) + \frac{d_i}{C_n}|x|^\alpha\right) d_i W_{i,n}(x)  \leq \left( d_1 W_{1,n}(x) + \frac{2d_1}{C_n}|x|^\alpha\right) H_n(x).
\]
Substituting in the integral estimate we find
\[
	\frac{1}{4} M_n \delta^3 \int_{B_R(x_0)} \left(H_n- d_1 W_{1,n}(x) - \frac{2d_1}{C_n}|x|^\alpha \right) H_n^2 \leq  C \left(1+ R^{N-1+2\alpha}\right).
\]
Passing to the limit in $n$ we find again a contradiction.
\end{proof}

Let now $M$ be any finite positive limit of $M_n$. Again, without loss of generality, we may assume that $M = 1$.
\begin{proof}[Finale]
We consider now the sequence of functions
\[
  Q_n(x) := H_n(x) - d_1 W_{1,n}(0)
\]
We observe that, form the previous discussion, we already know that there exists a function $Q$ such that, \utss{}, $Q_n \to Q$ locally uniformly in $\Omega_\infty$. The function $Q$ is non constant and globally $\alpha$-H\"older continuous. 

By definition $H_n \geq d_1 W_{1,n}$ and $d_1 W_{1,n}$ converges locally uniformly to a constant. As a result $Q \geq 0$ in $\Omega_\infty$. Moreover we have
\begin{multline*}
   \sum_{i=1}^{N_n} d_i^2 W_{i,n}^2 - H_n^2 \leq \left(d_1 W_{1,n}(0) + \frac{1}{\delta C_n}|x|^{\alpha} \right)H_n - H_n^2 \\
   \leq H_n  \left( \frac{1}{\delta  C_n}|x|^{\alpha} - Q_n\right) \leq  (Q_n + d_{1} W_{1,n}(0)) \left(\frac{1}{\delta C_n}|x|^{\alpha} - Q_n\right) \\
   \leq - Q_n^2 +  \left(\frac{1}{\delta C_n}|x|^{\alpha} - d_1 W_{1,n}(0)\right)Q_n + \frac{ d_1}{\delta C_n}|x|^{\alpha}  W_{1,n}(0).
\end{multline*}
Similarly one can prove
\[
   \sum_{i=1}^{N_n} d_i^2 W_{i,n}^2 - H_n^2 \geq - Q_n^2 - \left(\frac{1}{\delta C_n}|x|^{\alpha} + d_1 W_{1,n}(0)\right)Q_n + \frac{ d_1}{\delta  C_n}|x|^{\alpha}  W_{1,n}(0).
\]
Thus, passing to the limit in \eqref{lem diff ineq Hn}, we find that $Q$ solves
\[
  - \delta^{-3} Q^2 \leq -\Delta Q \leq - \delta^3 Q^2 \qquad \text{in $\R^n$.}
\]
By \cite[Lemma 2]{BrezisInfinity}, we infer that necessarily $Q \leq 0$, that is, under our assumptions, $Q \equiv 0$. We have reached again a contradiction.
\end{proof}

We have thus reached a contradiction with the blow-up assumption. This conclude the proof of the $C^{0,\alpha}$ uniform bounds for $\mathbf{w}$ and the proof of Theorem \ref{prp asymptotic k}.

\section{The limit system is a finite dimensional system - Proof of Theorem \ref{prp sing lim}}

Once the uniform estimates in $\beta$ and $N$ are established, we can pass to the limit as $\beta \to +\infty$. For any sequence $(\mathbf{w}_n, u_n)$ of solutions of \eqref{eqn model k}, defined for $\beta_n \to +\infty$ (here we make no assumption on $N_n$), we introduce the limit  vector $(\bar{\mathbf{w}}, \bar u) \in C^{0,\alpha}(\overline{\Omega})  \times C^{2,\alpha}(\overline{\Omega})$ such that, \utss{},
\[
  \lim_{n \to +\infty} w_{i,n} = \bar w_{i} \quad C^{0,\alpha}(\overline{\Omega}), \forall i \qquad \text{and} \qquad \lim_{n \to +\infty} u_{n} = \bar u \quad C^{2,\alpha}(\overline{\Omega}).
\]
The existence of $(\bar{\mathbf{w}}, \bar u)$ is a consequence of Lemma \ref{lem lim bu WH}. Let $N$ stand now for the number of non-zero components of the vector $\bar{\mathbf{w}}$. We show that $N$ is bounded. In particular, as a result of \cite[Lemma 6.1 and Theorem 6.3]{BerestyckiZilio_PI}, we can give an explicit upper bound for $N$.
 
In the following we will be many concerned with the $\mathbf{w}$ components of the system. We shall deduce stronger compactness properties, and derive the system of differential inequalities verified at the limit of segregation. Regarding the component $u$, we immediately find
\begin{lemma}
We have that
\[
	\begin{cases}
		- \Delta \bar u = \left( \lambda - \mu \bar u - \sum_{i} k_i \bar w_i\right) \bar u &\text{ in $\Omega$}\\
		\partial_\nu \bar u = 0 &\text{ on $\partial \Omega$}.
	\end{cases}	
\]
\end{lemma}
\begin{proof}
This is a direct consequence of the a priori estimates of Theorem \ref{prp asymptotic k}.
\end{proof}

We start with the following results

\begin{lemma}\label{lem limits are h1}
We have that
\[
  \lim_{n \to +\infty} w_{i,n} = \bar w_i  \quad H^1(\Omega), \forall i \qquad \text{and} \qquad \lim_{n \to +\infty} u_{n} = \bar u \quad H^1(\Omega).
\]
Furthermore, the limit functions verify
\[
  \bar w_i \bar w_j \equiv 0 \quad \text{in $\Omega$}, \qquad \forall i \neq j.
\]
For any $i$ there exists a (non negative) measure $\mu_i$ such that
\[
  \lim_{n \to +\infty} \int_\Omega \beta_n w_{i,n} \sum_{j\neq i} a_{ij} w_{j,n} \eta = \int \eta d \mu_i \qquad \text{for all $\eta \in C^\infty(\overline{\Omega})$}.
\]
Morevoer $\mu_i \in (H^1(\Omega))'$ and each limit function $\bar w_i$ is a weak solution of
\[
  d_i \int_{\Omega }\nabla \bar w_{i} \cdot \nabla \eta + \int \eta d\mu_i = \int_{\Omega}  \left(-\omega_i + k_i \bar u \right) \bar w_{i} \eta \qquad \text{for all $\eta \in H^1(\Omega)$}.
\]
\end{lemma}
\begin{proof}
The convergence of the sequence $u_n$ is immediate consequence of the convergence in $C^{2,\alpha}$, thus we only need to consider the convergence of $w_{i,n}$. First of all, we show that the sequence is uniformly bounded in $H^1(\Omega)$ and thus convergence weakly in $H^1(\Omega)$ to $\bar w_i$. To do this, we multiply the equation in $w_{i,n}$ by $w_{i,n}$ itself and integrate by parts. Exploiting the uniform $L^\infty$ bound, we find
\[
  d_i \int_{\Omega} |\nabla w_{i,n}|^2 + \beta_n w_{i,n}^2 \sum_{j\neq i} w_{j,n} = \int_{\Omega} \left(-\omega_i + k_i u_n\right) w_{i,n}^2 \leq C.
\]
Here the constant $C$ is independent of $n$. Observe that this inequality already implies that $\bar w_i \bar w_j \equiv 0$ in $\Omega$ for all $i \neq j$, since $\beta_n \to +\infty$. Now, since boundedness in norm and pointwise convergence imply weak convergence, we find that $w_{i,n} \rightharpoonup \bar w_i$ in $H^1(\Omega)$. Let $\eta \in C^\infty(\overline{\Omega})$ be any smooth test function. Testing the equation against $\eta$, we find
\begin{equation}\label{eqn weak eta}
 \int_{\Omega} \left( \beta_n w_{i,n} \sum_{j\neq i} w_{j,n} \right) \eta = \int_{\Omega} - d_i \nabla w_{i,n} \cdot \nabla \eta + \left(-\omega_i + k_i u_n\right) w_{i,n} \eta.
\end{equation}
The right hand side is bounded uniformly and converges as $n \to +\infty$. Thus the linear functional in the left hand side converges weakly to a non negative measure of $\overline{\Omega}$, which we denote by $\mu_i$. Moreover, by passing to the limit in $n$ in the previous identity we find
\begin{equation}\label{eqn weak eta lim}
  d_i \int_{\Omega }\nabla \bar w_{i} \cdot \nabla \eta + \int \eta d\mu_i = \int_{\Omega}  \left(-\omega_i + k_i \bar u \right) \bar w_{i} \eta.
\end{equation}
Observe that by this identity we infer that not only $\mu_i$ is a non negative measure of $\overline{\Omega}$, but that $\mu_i \in (H^1(\Omega))'$. By taking $\eta \equiv 1$, we also find that there exists $C \geq 0$ such that 
\[
  \int_{\Omega} \left( \beta_n w_{i,n} \sum_{j\neq i} w_{j,n} \right) = \int_{\Omega} \left(-\omega_i + k_i u_n\right) w_{i,n}  \leq C \quad \text{and} \quad \mu_i(\overline{\Omega}) = \int_{\Omega} \left(-\omega_i + k_i \bar u \right) \bar w_{i} \leq C.
\]
Thus, subtracting \eqref{eqn weak eta} and \eqref{eqn weak eta lim} and taking $\eta = w_{i,n} - \bar w_i$, we find
\begin{multline*}
  \int_{\Omega} d_i \left| \nabla ( w_{i,n}-\bar w_{i}) \right|^2 = \int_\Omega \left(-\omega_i + k_i u_n\right) \left|w_{i,n} - \bar w_i\right|^2 \\
  - \int_{\Omega} \left( \beta_n w_{i,n} \sum_{j\neq i} w_{j,n} \right) ( w_{i,n}-\bar w_{i}) + \int ( w_{i,n}-\bar w_{i}) d \mu_i \leq C \left\| w_{i,n}-\bar w_{i} \right\|_{L^\infty(\Omega)}
\end{multline*}
and thus we obtain the strong convergence of the sequence in $H^1(\Omega)$.
\end{proof}

To go further in the analysis of the limit equation, we need a classical estimate \cite[Lemma 4.4]{ContiTerraciniVerzini_AdvMat_2005}. Here we show a stronger result that still follows by a similar argument.

\begin{lemma}\label{lem dec exp}
Let $\Omega \subset \R^n$ be a smooth domain. Let $x_0 \in \overline{\Omega}$, $0 < \rho < 1$ and $B_\rho = B_\rho(x_0)$. For $U, M > 0$ we consider $u \in H^1(\overline{B_\rho \cap \Omega})$ a non-negative subsolution of
\[
  \begin{cases}
    - \Delta u \leq - M u &\text{in $B_\rho \cap \Omega$}\\
    u \leq U &\text{in $B_\rho \cap \Omega$}\\
    \partial_\nu u \leq 0  &\text{on $B_\rho \cap \partial\Omega$.}
  \end{cases}
\]
There exist a constant $C = C(\Omega)$ such that, for $\rho$ small and $M$ large enough, 
\[
  u(x) \leq U e^{-C \rho \sqrt{M}} \qquad \text{for every $x \in \overline{B_{\rho/2}\cap \Omega}$.}
\]
\end{lemma}
\begin{proof}
We distinguish between two cases. If $B_\rho \subset \Omega$ then the result follows by \cite[Lemma 4.4]{ContiTerraciniVerzini_AdvMat_2005}. It suffices to consider a supersolution of the form 
\[
  v = U e^{\alpha (X^2-\rho^2)}
\]
for a suitable $\alpha > 0$. Thus we need only to consider the case $B_\rho \setminus \Omega \neq \emptyset$. We will present here a proof of this second case.

We recall that a domain $\Omega$ is smooth ($C^{2,\alpha}$) if, locally at its boundary, there exist diffeomorphisms  (of class $C^{2,\alpha}$) between $\Omega$ and a half-space $H = \{X\in \R^n : X_1 > 0\}$. Fixing a point $x_0$ close to the boundary and $\rho$ small, we let $F$ be such a diffeomorphism at $B_\rho \cap \Omega$. Among the possible choices of $F$, we assume here that $F(x_0) = 0$, $J_{F}(x)\nu = e_1$ for all $x \in B_\rho \cap \partial \Omega $, and, moreover, that there exists $\eps > 0$ small such that for any $0 < r \leq \rho$ it holds
\begin{equation}\label{eqn set are sphere}
  B_{(1-\eps)r}(x_0) - x_0 \subset F(B_r(x_0)) \subset B_{(1+\eps)r}(x_0) - x_0.
\end{equation}
This is always true if $\rho$ is sufficiently small and thus $x_0$ sufficiently close to the boundary.

Under these assumptions, we let $\bar u(X) = u \circ F^{-1}(X)$. We find that $\bar u$ verifies
\begin{equation}\label{eqn ubar diffeo}
  \begin{cases}
    - \div( A \nabla \bar u) + M \bar u \leq 0 &\text{in $F(B_\rho) \cap H$}\\
    \bar u \leq U &\text{in $F(B_\rho) \cap H$}\\
    \partial_\nu \bar u \leq 0  &\text{on $F(B_\rho) \cap \partial H$}
  \end{cases}
\end{equation}
where $A(X) = |\det J_{F^{-1}}(X)| J_{F^{-1}}^{-T}(X) J_{F^{-1}}^{-1}(X)$ and $J_{F^{-1}}(X)$ is the Jacobian matrix of the diffeomorphism $F^{-1}$ at $X = F(x)$. The matrix field $A$ is regular (more precisely, $C^{1,\alpha}$) and positive-definite at each point. We let
\[
  L = \sup_{X \in F(B_\rho) \cap H}  |\nabla A(X)| \quad \text{and} \quad \Lambda = \sup_{X \in F(B_\rho) \cap H, |\xi| = 1} A(X) \xi \cdot \xi.
\]
Observe that the constants $L$ and $\Lambda$ ultimately depend on $\Omega$ and, for $x_0$ close to the boundary and $\rho$ small, $L \to 0$ and $\Lambda \to 1$, since the diffeomorphisms locally convergence to an isometry.

We look for a super-solution in $B_{(1-\eps)\rho}(x_0) \cap H \subset F(B_\rho) \cap H$ of equation \eqref{eqn ubar diffeo} of the form
\[
  v = U e^{\alpha(X^2-(1-\eps)^2\rho^2)}
\]
with $\alpha > 0$. Clearly we have
\[
  \begin{cases}
     v = U &\text{in $\partial B_{(1-\eps)\rho}\cap H$}\\
    \partial_\nu v \geq 0  &\text{on $B_{(1-\eps)\rho} \cap \partial H$.}
  \end{cases}
\]
Concerning the differential equation, in order to have a super-solution in $B_{(1-\eps)\rho}\cap H$ we impose
\begin{multline*}
  \div(A\nabla v) - M v = \left\{2\alpha \left[ \div A \cdot X + n \tr A  \right] + 4 \alpha^2 AX \cdot X - M\right\} v\\
  \leq \left\{2\alpha \left[ n L (1-\eps) \rho + n^2 \Lambda \right] + 4 \alpha^2 \Lambda (1-\eps)^2\rho^2 - M\right\} v \leq 0.
\end{multline*}
If $M$ is large enough, the inequality is verified by
\[
  \alpha = \frac{1}{3(1-\eps)\rho} \sqrt{\frac{M}{\Lambda}}.
\]
By the comparison principle we find $\bar u \leq v$ in $B_{(1-\eps)\rho}\cap H$. For all $X \in F(B_{\rho/2}(x_0)) \cap H \subset B_{(1+\eps)\rho/2}\cap H$ this inequality reads
\[
  \bar u(X) \leq U e^{\frac{1}{3(1-\eps)\rho} \sqrt{\frac{M}{\Lambda}}(\frac14 (1+\eps)^2\rho^2-(1-\eps)^2\rho^2)} = U e^{-\frac{3-10\eps+3\eps^2}{12 (1-\eps)} \frac{1}{\sqrt{\Lambda}} \rho \sqrt{M}} = Ue^{-C\rho \sqrt{M}}.
\]
We conclude the proof transforming back to the original domain.
\end{proof}

\begin{lemma}\label{lem eq in support}
For any $i \in \N$, either $\bar w_i \equiv 0 $ or
\[
  \begin{cases}
    -d_i \Delta \bar w_i = (k_i \bar u - \omega_i) \bar w_i &\text{in $\{\bar w_i > 0\} \cap \Omega$}\\
    \bar w_i = 0 &\text{on $\partial \{\bar w_i > 0\} \cap \Omega$}\\
    \partial_\nu \bar w_i = 0  &\text{on $\{\bar w_i > 0\} \cap \partial \Omega$}.
  \end{cases}
\]
In particular, the measure $\mu_i$ is supported on $\partial \{\bar w_i > 0\}$.
\end{lemma}
\begin{proof}
Since $\bar w_i$ is H\"older continuous, the set $\{\bar w_i > 0\}$ is relatively open in $\overline{\Omega}$. If it is not empty, let $x_0 \in \{\bar w_i > 0\}$. There exists $r_0 > 0$ such that
\[
  \bar w_i(x) > \frac{1}{2} \bar w_i(x_0) \qquad \text{for all $x \in B_{r_0}(x_0) \cap \overline{\Omega}$}.
\]
We need to show that 
\[
  \mu_i\left(B_{r_0/2}\cap\overline{\Omega}\right) = 0.
\]
As $w_{i,n}$ converges uniformly to $\bar w_i$, for $n$ large enough we have that
\[
  w_{i,n}(x) > \frac{1}{4} \bar w_i(x_0) \qquad \text{for all $x \in B_{r_0}(x_0) \cap \overline{\Omega}$}.
\]
For $n$ large, we sum all the components of system \eqref{eqn model k} for $j \neq i$. Letting $h_{i,n} = \sum_{j\neq i} d_j w_{j,n}$, we find that the estimate
\begin{multline*}
   -\Delta h_{i,n} = \sum_{j \neq i } \left(-\omega_{j} + k u_n\right) w_{j,n} - \beta_n \sum_{j\neq i} \sum_{k \neq j} a_{jk} w_{j,n} w_{k,n} \\ 
   \leq C(1-\beta_n w_{i,n}) h_{i,n} \leq C\left(1-\frac14 \beta_n \bar w_{i}(x_0)\right) h_{i,n}
\end{multline*}
holds true in $\Omega$, and in particular in $B_{r_0}(x_0) \cap \overline{\Omega}$. Here $C > 0$ is a constant that can be chosen independently of $n$. We now recall that, by Proposition \ref{prp asymptotic k}, the sequence $h_{i,n}$ is bounded uniformly in $\Omega$. By Lemma \ref {lem dec exp} we find that there exist positive constants  $C, C'$ such that, for $n$ large enough
\[
	h_{i,n}(x) \leq C e^{-C' r_0 \sqrt{\beta_n \bar w_i(x_0)} } \qquad \text{for all $x \in \overline{B_{r_0/2}(x_0) \cap \Omega}$}.
\]
As a result, we have that
\[
	\sup_{x \in \overline{B_{r_0/2}(x_0) \cap \Omega}}\beta_n \sum_{j \neq i} a_{ij} w_{j,n}w_{i,n} \leq C \beta_n e^{-C' r_0 \sqrt{\beta_n \bar w_i(x_0)} } \to 0
\]
for $n$ that diverges to $\to +\infty$. We conclude the proof by plugging this estimate in the equation satisfied by $w_{i,n}$ and taking the limit in $n$ on the set $\overline{B_{r_0/2}(x_0) \cap \Omega}$.
\end{proof}

We have  established the limit equation satisfied by the densities $\bar w_i$. Using it, we can prove that, in the limit of segregation, only a finite number of densities can persists. First we point out that the sets $\{\bar w_i >0\}$ are relatively open subsets of $\overline{\Omega}$ and they are also disjoint, thus
\[
	\sum_{i} |\{\bar w_i > 0\}| \leq |\Omega|.
\]

\begin{lemma}\label{lem min vol}
Assume that the limit $\mathbf{\bar w}$  has at least two non-zero components. Let $i \in \N$ stand for the index of a component $\bar w_i$ such that the area of the set $\{\bar w_i > 0\}$ is at most  equal to half of the area of $\Omega$. Then, it holds
\[
  \left|\left\{\bar w_{i} > 0 \right\}\right|^{2/n} \geq C(\Omega) \frac{d_i \mu}{\lambda k_i - \mu \omega_i}
\]
where $C>0$ is a constant that depends only on $\Omega$. Consequently, there exists $\hat N \in \N$ such that at most $\hat N$ components of $\mathbf{w}$ are non zero.
\end{lemma}
The previous statement gives already an a priori estimate on the number $\bar N$. Indeed, it must be
\[
	\hat N \leq C(\Omega) \left(\max_{i}\frac{\lambda k_i - \mu \omega_i}{d_i \mu}\right)^{\frac{n}{2}}.
\]
\begin{proof}
Let $i \in \N$ be fixed, we consider the function $\bar w_i$ and the relatively open subset $\{\bar w_i > 0\}$ of $\overline{\Omega}$.
Combining the previous results, we find that $\bar w_i \in H^1_0(\{\bar w_i > 0\})$ satisfies
\[
  \begin{cases}
    -\Delta \bar w_i \leq \frac{\lambda k_i - \mu\omega_i}{d_i\mu} \bar w_i &\text{in $\{\bar w_i > 0\} \cap \Omega$}\\
    \bar w_i = 0 &\text{on $\partial \{\bar w_i > 0\} \cap \Omega$}\\
    \partial_\nu \bar w_i = 0  &\text{on $\{\bar w_i > 0\} \cap \partial \Omega$}.
  \end{cases}
\]
Here we have used the fact that $\bar u \leq \lambda / \mu$ pointwisely in $\Omega$. For the function $\bar w_i$ to be non zero, it must be that 
\[
	\lambda_1(\{\bar w_i > 0\}) \leq \frac{\lambda k_i - \mu\omega_i}{d_i\mu}
\]
where $\lambda_1$ is the first eigenvalue of the Laplace operator in $\{\bar w_i > 0\}$ with the same boundary conditions of $\bar w_i$
\[
	\lambda_1(\{\bar w_i > 0\}) = \min_{u \in H^1_0(\{\bar w_i > 0\})} \left. \int |\nabla u|^2 \right/ \int u^2.
\]
Since by assumption $|\{\bar w_i > 0\}| \leq \frac12 |\Omega|$, by the relative Faber-Krahn inequality (see \cite[Proposition 2.3]{SpectralDrop}), we find that
\[
	\lambda_1^{-1}(\{\bar w_i > 0\}) \leq C(\Omega) \left|\left\{\bar w_{i} > 0 \right\}\right|^{2/n}
\]
which yields the desired inequality.
\end{proof}

Once we have shown that in the segregation limit only a finite number of components of the vector $\mathbf{\bar w}$ can be non zero, we can reason in the same way as in \cite[Theorem 6.3]{BerestyckiZilio_PI} in order to give a more explicit, yet asymptotic, estimate on $\hat N$, by means of Weyl's asymptotic law for the eigenvalues of the Neumann Laplacian . Indeed, we find
\begin{lemma}
Let $(\mathbf{w},\bar u)$ be any limit of solutions of \eqref{eqn model k} when $\beta \to +\infty$. Let $\hat N$ be the number of components of $\mathbf{\bar w}$ that are not identically zero. It holds
\[
  \hat N \lesssim \frac{|\Omega|}{\left(\frac{\pi}{4} \right)^{\frac{n}{2}} \Gamma\left(\frac{n}{2}+1\right)}   \left(\max_{i} \frac{\lambda k_i - \mu \omega_i}{d_i \mu} \right)^{\frac{n}{2}}
\]
for $\max_{i} \frac{\lambda k_i - \mu \omega_i}{d_i \mu} \to +\infty$.
\end{lemma}

To conclude the proof of Theorem \ref{prp sing lim} we have to establish the uniform Lipschitz bounds of the components $\mathbf{w}_\beta$. For this result we need some properties on the structure of the limit free-boundary. We thus consider now Theorem \ref{prp free boundary} and then come back later to conclude the proof of Theorem \ref{prp sing lim}.

\section{Structure of the limit free-boundary problem - Proof of Theorem \ref{prp free boundary}}

We have shown that the limit segregation problem has at most a finite number of non-zero components. We have also derived the limit equation satisfied by each non-zero component of the limit system. In order to study the limit free-boundary, we need an additional result about the complementary conditions.

\begin{lemma}\label{lem compl cond}
For every $i \in \mathbb{N}$ and every $\eta \in H^1(\Omega)$ non negative test function, the following inequality holds
\[
	\int \nabla \left( d_i \bar w_i-\sum_{j \neq i} d_j \bar w_j \right) \nabla \eta \geq \int \left[(- \omega_i + k_i \bar u )\bar w_i-\sum_{j \neq i} (- \omega_j + k_j \bar u ) \bar w_j \right]\eta.
\]
\end{lemma}
\begin{proof}
It suffices to consider the equation satisfied by the function $d_i w_{i,n}-\sum_{j \neq i} d_j w_{j,n}$ and then pass to the limit in $n$. We find
\[
	- \Delta \left( d_i w_{i,n}-\sum_{j \neq i} d_j w_{j,n} \right) = (- \omega_i + k_i \bar u_n ) w_{i,n}-\sum_{j \neq i} (- \omega_j + k_j u_n ) w_{j,n} + \beta \sum_{j \neq i} \sum_{k \neq j,i} a_{j,k} w_{j,n} w_{k,n}.
\]
Let now $\eta \in C^{\infty}(\Omega)$ be a non-negative test function. Multiplying the previous equation by $\eta$ and integration by parts we find
\begin{multline*}
	\int_{\Omega} - \Delta \eta  \left( d_i w_{i,n}-\sum_{j \neq i} d_j w_{j,n} \right) + \int_{\partial \Omega} \partial_\nu \eta  \left( d_i w_{i,n}-\sum_{j \neq i} d_j w_{j,n} \right) \\
	\geq \int_{\Omega} \left[(- \omega_i + k_i \bar u_n ) w_{i,n}-\sum_{j \neq i} (- \omega_j + k_j u_n ) w_{j,n}\right] \eta.
\end{multline*}
We then pass to the uniform limit in $n$. To conclude, we recall that at the limit, only a finite number of components of $\mathbf{\bar w}$ are non-zero, and these functions belong to $H^1(\Omega)$.
\end{proof}

We are now in a position to invoke classical results on the regularity and the structure of free-boundaries of segregation models. We have
\begin{theorem}
The common nodal set $\mathfrak{N} := \{x \in \overline{\Omega}: \sum_{i} \bar w_i = 0\}$ is a rectifiable set of Hausdorff dimension $n-1$. More precisely $\mathfrak{N}$ can be written as the disjoint union of two sets, $\mathfrak{R}$ and $\mathfrak{S}$, such that $\mathfrak{R}$ is relatively open and made of the union of a finite number of $\C^{1,\alpha}$ smooth sub-manifolds, while $\mathfrak{S}$ is a set of Hausdorff dimension $n-2$. Moreover $\mathfrak{R}$ meets orthogonally the boundary $\partial \Omega$ and $\mathfrak{N} \cap \partial \Omega$ is a set of Hausdorff dimension $n-2$, that can be decomposed as the disjoint union of a regular part (finite union of  $\C^{1,\alpha}$ smooth sub-manifolds of codimension $2$) and a singular part. 
\end{theorem}

Observe that the estimates of the codimension of the singular part, both inside of $\Omega$ and at he boundary $\partial \Omega$ are in general sharp for this kind of free-boundary problems.

\begin{proof}
We can directly apply the results in \cite[Theorem 16]{CaffKarLin} and \cite[Theorems 8.4 and 1.1]{TaTe} to conclude about the structure of $\mathfrak{N} \cap \Omega$.

Concerning the regularity of the common nodal set $\mathfrak{N}$ at the boundary $\partial \Omega$, we can proceed similarly as in Lemma \ref{lem dec exp}. Let $x_0 \in \partial \Omega$ and let $\rho > 0$ be small enough, so that there exists a $C^{2,\alpha}$ diffeomorphisms $F$ of $B_\rho(x_0) \cap \Omega$ into $B_\rho \cap H$, where $H = \{X \in \R^n : X_1>0\}$ is a half-space. Again, among the possible choices of $F$, we assume here that $F(x_0) = 0$, $J_{F}(x)\nu = -e_1$ for all $x \in B_\rho(x_0) \cap \partial \Omega $. Here $J_F(x)$ is the Jacobian matrix of $F$ computed at $x$. We also assume that $\partial_{X_1} J_{F_{-1}}(X) = 0$ for all $X \in B_\rho \cap \{X \in \R^n : X_1=0\}$. We consider the system of inequalities satisfied by $\mathbf{\hat w} = \mathbf{\bar w} \circ F$ and $\hat u = u \circ F$. By extending evenly across $B_\rho \cap \partial H$ the functions $\mathbf{\hat w}$ and $\hat u$, we find that they solve the system
\[
	\begin{dcases}
		- d_i \div(A\nabla \hat w_i)  \leq |\det J_{F^{-1}}|(- \omega_i + k_i \hat u -\mu_i) \hat w_i \\
		- \div\left[A\nabla \left( d_i \hat w_i-\sum_{j \neq i} d_j \hat w_j \right)\right] \geq |\det J_{F^{-1}}| \left[(- \omega_i + k_i \hat u )\hat w_i-\sum_{j \neq i} (- \omega_j + k_j \hat u ) \hat w_j\right] &\text{in $B_\rho$}\\
		- D \div(A\nabla \hat u) = |\det J_{F^{-1}}| \left(\lambda - \mu \hat u - \sum_{i} k_i \hat w_i \right)\hat u
	\end{dcases}
\]
where 
\[
	A(X) := \begin{cases}
			\left(|\det J_{F^{-1}}| J_{F^{-1}}^{-T} J_{F^{-1}}^{-1}\right)(X_1, X_2, \dots, X_n) &\text{for $X \in B_\rho \cap \{X \in \R^n : X_1 \geq 0\}$}\\
			\left(|\det J_{F^{-1}}| J_{F^{-1}}^{-T} J_{F^{-1}}^{-1}\right)(-X_1, X_2, \dots, X_n) &\text{for $X \in B_\rho \cap \{X \in \R^n : X_1 < 0\}$}
		\end{cases}
\]
and similarly for the other functions. Observe that $A$ is a $C^{1,\alpha}$ function (this is true thanks to our assumptions on $F$ and the regularity of the boundary of $\Omega$). The validity of this system of inequalities in $B_\rho$ can be shown following the same ideas as  Lemmas \ref{lem limits are h1}, \ref{lem eq in support} and \ref{lem compl cond}. 

Let $\mathfrak{\hat N}$ be the common nodal set of the functions $\mathbf{\hat w}$, $\mathfrak{\hat N} = \{x\in B_\rho : \mathbf{\hat w} = 0 \}$. Locally at the point $x_0$, the structure of the nodal set $\mathfrak{N}$ is related to the of the nodal set of $\mathbf{\hat w}$. 
We can then apply the regularity theory developed in \cite[Theorems 7.1 and 8.4]{TaTe} to conclude that $\mathfrak{\hat N}$ is a rectifiable set.

We now show that, since the functions are symmetric with respect to the plane $H$, the Hausdorff dimension of the set $\mathfrak{R} \cap \partial \Omega$ is $n-2$. Indeed, let us assume that $x_0 \in \partial \Omega \cap \mathfrak{N}$, so that $0 \in \mathfrak{\hat N}$. We only need to consider the case in which $0$ is in the regular part of $\mathfrak{\hat N}$. This implies in particular that there exists $R>0$ such that $\mathfrak{\hat N} \cap B_R$ is a smooth surface of codimension $1$, symmetric with respect to the plane $H$. Since this is true for any point of the regular part of $\mathfrak{\hat N}$ in $H$, we find the conclusion. 
\end{proof}

\section{Maximal number of components - Proof of Theorem \ref{thm max packs}}

We conclude by giving an extension of Lemma \ref{lem min vol}, in the same spirit of \cite[Theorem 4.3]{BerestyckiZilio_PI}. We use the structure of the limit free-boundary to extend the upper-bound on the number of non-zero components of $\mathbf{\bar w}$ also in the case of the system \eqref{eqn model k} with $\beta$ finite but large.

The main difference with respect to \cite[Theorem 4.3]{BerestyckiZilio_PI} is that here we do not assume any a priori bound on the number $N$ of non zero components. We can do this by exploiting the uniform a priori estimate of Theorem \ref{prp asymptotic k}. This is a delicate but technical detail that does not change drastically the proof. Nevertheless, we have decided to include the proof here in an abridged form, for the sake of completeness. 

We start with a result stating that the zero solution $\mathbf{v} = (\mathbf{w}, u) \equiv 0$ is isolated. In particular, it implies that any sequence of solutions $\mathbf{v}_n = (\mathbf{w}_n, u_n)$ such that $u_n \to 0$ uniformly, is eventually constant and equal to the zero solution.

\begin{lemma} 
There exists $\eta > 0$, independent of $\beta$ and $N$, such that if $\mathbf{v} = (\mathbf{w}, u)$ is a non negative solution of \eqref{eqn model k} with $0\leq u < \eta$ in $\Omega$, then $\mathbf{v} \equiv 0$.
\end{lemma}
\begin{proof}
Let $\eta = \delta^2$, where $\delta> 0$ is the constant in \eqref{unifass}. Then $\delta^2 \leq \omega_i / k_i$ for all possible choice of coefficients. If $u < \eta$ in $\Omega$, then from \eqref{eqn model k} we find
\[
	\begin{cases}
		- d_i \Delta w_{i}  = \underbrace{\left(- \omega_i + k_i u - \beta \sum_{j \neq i} a_{ij} w_{j} \right)}_{<0} w_i  &\text{in $\Omega$}\\
		\partial_\nu w_{i} = 0 &\text{on $\partial \Omega$}
	\end{cases}
\]
which yields $w_{i} \equiv 0$ for all $i$. As a result, $u$ solves the (logistic) equation
\[	
  \begin{cases}
		- D \Delta u = \left(\lambda - \mu u \right)u   &\text{in $\Omega$}\\
		\partial_\nu u = 0 &\text{on $\partial \Omega$}.
	\end{cases}
\]
The maximum principle prescribes that the only non negative solutions of this equation are $u \equiv 0$ and $u \equiv \lambda / \mu$. Thus, again by \eqref{unifass}, we conclude that it must be $u \equiv 0$.
\end{proof}

We now show a similar result concerning the non zero solution. We prove that if a sequence of solutions $\mathbf{v}_n$ converge to a non zero limit $\mathbf{v}$ for $\beta_n \to +\infty$, then, for $\beta_n$ large, the number of non zero components of each $\mathbf{v}_n$ must be at most the same of its limit.

\begin{lemma} 
For any $N \geq 1$ fixed, there exists $\bar \beta = \bar \beta(N) > 0$ such that the following statement holds. Let  $\mathbf{v} = (\mathbf{w}, u)$ be any solution of \eqref{eqn segr model} such that $\mathbf{w}$ has $N$ non zero components. Let $\mathbf{v}_n = (\mathbf{w}_{n}, u_{n})$ be a family of solutions of \eqref{eqn model k} with $\beta_n \to +\infty$ and such that $\mathbf{w}_n \to \mathbf{w}$ component-wise in $C^{0,\alpha}\cap H^1(\Omega)$, $u_n \to u$ in $C^{2,\alpha}(\Omega)$. Whenever $\beta_n > \bar \beta$, $\mathbf{w}_n$ has at most $N$ non zero components.
\end{lemma}

\begin{proof}
We argue by contradiction, assuming that there exists a sequence of solutions $\mathbf{v}_n$ such that $\mathbf{w}_n$ has at least $N+1$ non zero components for all $n\in \N$ but converges to a limit $\mathbf{v}$ such that $\mathbf{w}$ has only $N \geq 1$ non zero components. We first observe that $u$, and thus $u_n$, is necessarily strictly positive.

Up to striking out a subsequence, we relabel the components so that the first $N$ components of $\mathbf{w}_n$ converge in $C^{0,\alpha}\cap H^1(\Omega)$ to the non zero components of $\mathbf{w}$, while the other components converge to $0$ and ${w}_{N+1,n} > 0$ for all $n$. We consider the sequence of functions
\[
  \bar {w}_{N+1,n} = \frac{w_{N+1,n}}{\|w_{N+1,n}\|_{L^\infty(\Omega)}}.
\]
These functions solve, for any $n$,
\begin{equation}\label{eqn model k scaled}
	\begin{cases}
		- d_{N+1}\Delta \bar w_{N+1,n}  = \left(- \omega_{N+1} + k_{N+1} u_{n} - \beta \sum_{j \neq {N+1}}w_{j,n}\right) \bar w_{N+1,n} &\text{ in $\Omega$}\\
		\partial_\nu \bar w_{N+1,n} = 0 &\text{ on $\partial \Omega$}.
	\end{cases}
\end{equation}
We recall that, by Theorem \ref{prp asymptotic k}, $u_n$ is uniformly bounded from above by $\lambda / \mu$. We find that
\[
	\int_{\Omega} |\nabla \bar w_{N+1,n}|^2 \leq \frac{\lambda k_{N+1} - \mu \omega_{N+1}}{\mu d_{N+1}} |\Omega| < \frac{1}{\delta^4} |\Omega|.
\]
This estimate, together with $\|\bar w_{N+1,n}\|_{L^\infty} = 1$, implies that $\bar w_{N+1,n}$ is uniformly bounded in $H^1(\Omega)$. Let $\bar w_{N+1} \in H^1(\Omega)$ be any weak limit of this sequence. The previous uniform estimates and Sobolev's embedding theorem assure us that $\bar w_{N+1,n} \to \bar w_{N+1}$ strongly in $L^p(\Omega)$ for any $p < \infty$. Moreover, from \eqref{eqn model k scaled} we can show that $\bar w_{N+1}$ is not identically zero. Indeed, first we see that 
\[
	\begin{cases}
		- d_{N+1} \Delta \bar w_{N+1,n} \leq \left(- \omega_{N+1} + k_{N+1} u_{n} \right) \bar w_{N+1,n} \\
		\partial_\nu \bar w_{N+1,n} = 0 &\text{ on $\partial \Omega$}.
	\end{cases}
\]
Let $g_{n} \in H^1(\Omega)$ be a solution to 
\[
	\begin{cases}
		- d_{N+1} \Delta g_{n} + \omega_{N+1} g_{n} = k_{N+1} u_{n} \bar w_{{N+1},n} \\
		\partial_\nu g_{n} = 0 &\text{ on $\partial \Omega$}.
	\end{cases}
\]
The comparison principle states that $0 \leq \bar w_{N+1,n} \leq g_{n}$. Standard regularity estimates give us
\[
	\|g_{n}\|_{\C^{0,\alpha}(\Omega)} \leq C \|g_{n}\|_{W^{2,p}(\Omega)} \leq C' \|\bar w_{N+1,n} \|_{L^p(\Omega)}
\]
for any $N/2 < p < \infty$ and suitable $C, C'$ and $\alpha > 0$. As a result, we have
\[
	1 =  \|\bar w_{N+1,n} \|_{L^\infty(\Omega)} \leq \|g_{n}\|_{L^\infty(\Omega)} \leq \|g_{n}\|_{\C^{0,\alpha}(\Omega)} \leq C'\|\bar w_{N+1,n} \|_{L^p(\Omega)}.
\]
Thus $\bar w_{N+1}$ is not identically zero. Let us use this information in order to reach a contradiction. For any $\eps > 0$, we consider the sets
\[
	\mathfrak{P}_{\eps} := \left\{ x \in \Omega : \sum_{i=1}^N w_{i}(x) > \eps \right\}.
\]
Clearly one has that $\mathfrak{N}\subset \Omega \setminus \mathfrak{P}_{\eps}$ and $\mathfrak{N} = \cap_{\eps>0} (\Omega \setminus \mathfrak{P}_{\eps})$ for $\eps$ small enough. We point out that, if $N=1$, then $\mathfrak{P}_{\eps} = \Omega$. By the uniform convergence of $(w_{1,n}, \dots, w_{N,n})$, we see that for any $\eps > 0$ there exists $n_{\eps} \in \N$ such that
\[
  \sum_{i=1}^N w_{i,n}(x) > \frac{\eps}{2} \qquad \forall x \in \Omega \cap \mathfrak{P}_{\eps}, n \geq n_{\eps}.
\]
Similarly, the compactness properties of the sequence $\bar w_{N+1,n}$ insure us that for any $\eps >0$ small enough there exists, for $n \in \N$ sufficiently large, 
\[
	\int_{\mathfrak{P}_{\eps} } \bar w_{N+1,n}^2 \geq \frac12 \int_{\Omega} \bar w_{N+1}^2.
\]
From \eqref{eqn model k scaled}, multiplying by $\bar w_{N+1,n}$ and integrating by parts, we get
\[
	\int_{\Omega} \left[ d_{N+1} |\nabla \bar w_{N+1,n}|^2 + \beta_n \left(\sum_{i=1}^N w_{i,n}\right) \bar w_{N+1,n}^2 \right]
	\leq \int_\Omega (-\omega_{N+1} + k_{N+1} u_{n}) \bar w_{N+1,n}^2.
\]
Thus, estimating from below and above the two sides, we obtain 
\[
	\frac{\eps \beta_n}{4} \int_{\Omega} \bar w_{N+1}^2 \leq \beta_n \frac{\eps}{2} \int_{\mathcal{P}_\eps} \bar w_{N+1,n}^2 \leq \dots \leq \frac{\lambda k_{N+1} - \mu \omega_{N+1}}{\mu} |\Omega|
\]
for any $\eps > 0$ and $n \geq n_\eps$. Thus we conclude that $\bar w_{N+1} \equiv 0$, a contradiction.
\end{proof}

We are in a position to conclude the proof of Theorem \ref{thm max packs}
\begin{proof}
From Lemma \ref{lem min vol} we know that, for $\beta \to +\infty$, any accumulation point $\mathbf{v} = (\mathbf{w},u)$ of solutions of \eqref{eqn model k} has at most $\hat N \geq 1$ non zero components $\mathbf{w}$. By the previous two results we know that there exists $\bar \beta = \min \{\bar \beta(N) : N = 1, \dots, \hat N \} > 0$ such that any solution $\mathbf{v}_\beta$ of \eqref{eqn model k} with $\beta > \bar \beta$ 
\begin{itemize}
  \item is the zero solution $\mathbf{v}_\beta = (0,0)$;
  \item has at most $\bar N$ non zero components of $\mathbf{w}_\beta$;
  \item converges component-wise to the solution $(0,\lambda/\mu)$, that is
  \[
    \lim_{\beta\to+\infty} \left(\max_{i=1, \dots,N} \|w_{i,\beta}\|_{C^{0,\alpha}\cap H^1} + \left\|u_\beta-\frac{\lambda}{\mu}\right\|_{C^{2,\alpha}} \right)=0.
  \]
\end{itemize}
This concludes the proof of Theorem \ref{thm max packs}.
\end{proof}

\bigskip
\noindent \textbf{Acknowledgements:}  This work has been supported by the ERC Advanced Grant 2013 n. 321186 ``ReaDi -- Reaction-Diffusion Equations, Propagation and Modelling'' held by Henri Berestycki, and by the French National Research Agency (ANR), within  project NONLOCAL ANR-14-CE25-0013.

\end{document}